\documentclass[12pt]{amsart}

\usepackage[twoside=false]{geometry}
\usepackage{amsmath,amssymb,amsthm,amscd}
\usepackage{mathrsfs}
\usepackage{dsfont}
\usepackage{lmodern}
\usepackage{mathtools}
\usepackage{enumerate}
\usepackage{graphicx}
\usepackage{hyperref}

\newtheorem{thmintro}{Theorem}

\newtheorem{corintro}[thmintro]{Corollary}
\newtheorem{theorem}{Theorem}[section]
\newtheorem{corollary}[theorem]{Corollary}
\newtheorem{lemma}[theorem]{Lemma}
\newtheorem{prop}[theorem]{Proposition}

\theoremstyle{definition}
\newtheorem{remark}[theorem]{Remark}
\newtheorem{example}[theorem]{Example}

\newcommand{\CC}{\mathbb{C}}

\newcommand{\NN}{\mathbb{N}}
\newcommand{\QQ}{\mathbb{Q}}

\newcommand{\ZZ}{\mathbb{Z}}

\newcommand{\LLL}{\mathcal{L}}

\newcommand{\WW}{\mathcal{W}}

\newcommand{\g}{\mathfrak{g}}

\newcommand{\sll}{\mathfrak{sl}}

\newcommand{\hh}{\mathfrak{h}}

\newcommand{\nn}{\mathfrak{n}}

\newcommand{\inv}{^{-1}}

\newcommand{\la}{\langle}
\newcommand{\ra}{\rangle}
\newcommand{\co}{\colon\thinspace}
\newcommand{\bra}{[\![}
\newcommand{\ket}{]\!]}

\DeclareMathOperator{\ad}{ad}

\DeclareMathOperator{\Aut}{Aut}

\DeclareMathOperator{\GL}{GL}

\DeclareMathOperator{\height}{ht}

\DeclareMathOperator{\supp}{supp}

\numberwithin{equation}{section}

\begin{document}

\renewcommand{\proofname}{{\bf Proof}}

\title{On the structure of Kac--Moody algebras}
\author{Timoth\'ee \textsc{Marquis}$^*$}
\thanks{$^*$ F.R.S.-FNRS post-doctoral researcher.} 
\address{Universit\'e catholique de Louvain, IRMP, Chemin du Cyclotron 2, bte L7.01.02, 1348 Louvain-la-Neuve, Belgique}
\email{timothee.marquis@uclouvain.be}

\begin{abstract}
Let $A$ be a symmetrisable generalised Cartan matrix, and let $\g(A)$ be the corresponding Kac--Moody algebra. In this paper, we address the following fundamental question on the structure of $\g(A)$: given two homogeneous elements $x,y\in\g(A)$, when is their bracket $[x,y]$ a nonzero element? As an application of our results, we give a description of the solvable and nilpotent graded subalgebras of $\g(A)$.
\end{abstract}

\maketitle

\section{Introduction}
By a theorem of J.-P.~Serre (\cite{Se66}), any finite-dimensional complex semisimple Lie algebra admits a presentation whose parameters are the entries of some matrix of integers $A$, called a \emph{Cartan matrix}. This presentation still makes sense if one allows more general integral matrices $A=(a_{ij})_{i,j\in I}$ ($I$ a finite set), called \emph{generalised Cartan matrices} (GCM). The corresponding Lie algebras $\g(A)$ (the \emph{Kac--Moody algebras}) were introduced independently in 1967 by V.~Kac (\cite{Kac67}) and R.~Moody (\cite{Mo67}). From a mere generalisation to infinite dimension of the semisimple Lie algebras (which are the Kac--Moody algebras of \emph{finite type}), Kac--Moody algebras soon became central objects of study, with a wide array of applications in a variety of mathematical domains, as well as in theoretical physics (see e.g. \cite{Kac}). Over the past years, specific Kac--Moody algebras of \emph{hyperbolic} and \emph{Lorentzian type} have also been intensively investigated in connection to \emph{$M$-theory}  (see e.g. \cite{Damour02}, \cite{Damour05}, \cite[\S 19]{FGKP18}, and also \cite{We01}, \cite{GOW02}). On the other hand, understanding the structure of Kac--Moody algebras is an essential step towards a better understanding of the corresponding \emph{Kac--Moody groups}; these groups, introduced in the late 1980's, turn out to exhibit a very rich structure, paralleling that of semisimple algebraic groups, and have become over the last decades prominent objects of study in many different areas, including geometric group theory, algebraic geometry and representation theory (see e.g. \cite{Kumar}, \cite[\S 3.1]{Cap16}, \cite{KMGbook}).

Kac--Moody algebras share many properties with their finite-dimensional sisters; in particular, they possess a \emph{root space decomposition}
$$\g(A)=\hh\oplus\bigoplus_{\alpha\in\Delta}\g_{\alpha}$$
with respect to the adjoint action of a \emph{Cartan subalgebra} $\hh$, with associated set of \emph{roots} $\Delta\subseteq\hh^*$, as well as a \emph{triangular decomposition}
$$\g(A)=\nn^-\oplus\hh\oplus\nn^+,$$
where $\nn^{\pm}:=\bigoplus_{\alpha\in\Delta^{\pm}}\g_{\alpha}$ is the subalgebra of $\g(A)$ associated to the set of \emph{positive/negative} roots $\Delta^{\pm}=\Delta\cap\pm\sum_{i\in I}\NN\alpha_i$ with respect to a set $\{\alpha_i \ | \ i\in I\}$ of \emph{simple} roots.
Moreover, if $A$ is \emph{symmetrisable} (a mild assumption made throughout this paper, see \S\ref{section:preliminaries} for precise definitions), then $\g(A)$ admits a non-degenerate invariant form $(\cdot|\cdot)$ that generalises the Killing form; the restriction to $\hh$ of this form is non-degenerate, and induces a bilinear form
$$\hh^*\times\hh^*\to\CC: \quad (\alpha,\beta)\mapsto (\alpha|\beta)$$
on $\hh^*$.

On the other hand, Kac--Moody algebras also show some striking differences: while some roots in $\Delta$ have analoguous properties to the roots of a semisimple Lie algebra (such roots are called \emph{real}), the key novelty of Kac--Moody algebras (of non-finite type) is the apparition of \emph{imaginary} roots, with a completely different behaviour. The sets $\Delta^{re}$ and $\Delta^{im}$ of real and imaginary roots can be described as 
$$\Delta^{re}=\{\alpha\in\Delta \ | \ (\alpha|\alpha)>0\}\quad\textrm{and}\quad \Delta^{im}=\{\alpha\in\Delta \ | \ (\alpha|\alpha)\leq 0\}.$$
One of the most notable differences between real and imaginary roots concerns the dimension of the corresponding root spaces: while $\dim\g_{\alpha}=1$ for all $\alpha\in\Delta^{re}$, the dimensions of the root spaces $\g_{\alpha}$ with $\alpha\in\Delta^{im}$ might be arbitrarily large, and determining these root multiplicities is still a widely open problem (see \cite{CFL14} for a recent survey of the state of the art). In fact, despite a considerable volume of works on the topic, the structure of general Kac--Moody algebras remains, to this day, largely mysterious.

One exception is the case where all imaginary roots $\alpha\in\Delta^{im}$ are \emph{isotropic}, in the sense that $(\alpha|\alpha)=0$. The corresponding Kac--Moody algebras, of so-called \emph{affine type}, have concrete realisations as (twisted) loop algebras over semisimple Lie algebras (or rather, suitable extensions thereof), and their structure is thus well-understood. In particular, $\{\dim\g_{\alpha} \ | \ \alpha\in\Delta\}$ is in that case bounded. On the other hand, when $\g(A)$ is of \emph{indefinite type}, i.e. neither of finite nor of affine type, the set $\{\dim\g_{\alpha} \ | \ \alpha\in\Delta\}$ is unbounded, and there is not a single instance where a ``concrete realisation'' of $\g(A)$, as in the affine case, is known.

In order to elucidate the structure of general Kac--Moody algebras beyond the foundational results of the theory (see \cite{Kac}), the efforts of the Kac--Moody community have essentially been focussed on obtaining root multiplicity formulas. Pioneering works by several authors (notably, \cite{BM79}, \cite{FF83} and \cite{Ka94}) led to several such formulas in closed form, at various levels of generality, and these formulas were applied in a number of papers to determine explicitely the root multiplicities of ``small'' roots for some particular Kac--Moody algebras. However, these formulas are of very little help in understanding how $\dim\g_{\alpha}$ varies when $\alpha\in\Delta^{im}$ varies and, \emph{a fortiori}, in getting global information on the Lie algebra structure of $\g(A)$. In fact, apart from a monotonicity result, obtained in \cite[Proposition~5.6]{KM95}, stating that $\dim\g_{\alpha}\leq\dim\g_{\alpha+\alpha_1+\alpha_2}$ for any root $\alpha\in\Delta^+$ associated to the GCM $A=\begin{psmallmatrix}2 & -a\\ -a &2\end{psmallmatrix}$ ($a\in\NN$), there seems to be no general result beyond \cite{Kac} that provides information on the Lie bracket of $\g(A)$, or even that offers some comparison results on the root multiplicities.

\medskip

In the present paper, we take a different approach, by addressing the following fundamental question on the Lie algebra structure of $\g(A)$:
$$\textbf{\emph{Given $x\in\g_{\alpha}$ and $y\in\g_{\beta}$, when is $[x,y]$ a nonzero element?}}$$
Our main theorem is as follows.

\begin{thmintro}\label{thmintro:main}
Let $\alpha,\beta\in\Delta$. If $(\alpha|\beta)<0$ then $[x,y]\neq 0$ for all nonzero $x\in\g_{\alpha}$ and $y\in\g_{\beta}$, unless $\alpha=\beta$ and $\CC x=\CC y$.
\end{thmintro}

As the condition $(\alpha|\beta)<0$ is almost always satisfied for positive imaginary roots $\alpha,\beta\in\Delta^{im+}:=\Delta^{im}\cap\Delta^+$ (see Lemma~\ref{lemma:basic_propI}), Theorem~\ref{thmintro:main} allows in particular for a precise description of the Lie bracket on the \emph{imaginary subalgebra} (see \S\ref{subsection:WGOgA})
$$\nn^{im+}:=\bigoplus_{\alpha\in\Delta^{im+}}\g_{\alpha}$$
of $\nn^+$ (note that there is a \emph{Chevalley involution} $\omega\in\Aut(\g(A))$ exchanging $\nn^+$ and $\nn^-$, whence our focus on $\nn^+$).

\begin{corintro}\label{corintro:main}
Let $\alpha,\beta\in\Delta^{im+}$. Then one of the following holds:
\begin{enumerate}
\item
$[\g_{\alpha},\g_{\beta}]=\{0\}$. This occurs if and only if either $\alpha+\beta\notin\Delta$, or $\alpha,\beta$ are proportional isotropic roots.
 \item
 $[x,y]\neq 0$ for every nonzero $x\in\g_{\alpha}$ and $y\in\g_{\beta}$ such that $\CC x\neq\CC y$. 
 \end{enumerate}
\end{corintro}
Note that the case (1) in Corollary~\ref{corintro:main} is completely understood, as we also determine precisely the pairs $\alpha,\beta\in\Delta^{im+}$ such that $\alpha+\beta\notin\Delta$ (see Lemma~\ref{lemma:imaginary_not_connected}). The case (2), on the other hand, implies the following dramatic generalisation of \cite[Proposition~5.6]{KM95} to arbitrary symmetrisable GCM and arbitrary pairs of positive roots $(\alpha,\beta)$ (Proposition~5.6 in \emph{loc.cit.} covers the case $\beta=\alpha_1+\alpha_2$ for the symmetric $2\times 2$ GCM).

\begin{corintro}\label{corintro:dimension}
Let $\alpha,\beta\in\Delta^{im+}$ with $\alpha\neq\beta$ be such that $[\g_{\alpha},\g_{\beta}]\neq\{0\}$. Then 
 $$\dim [\g_{\alpha},\g_{\beta}]\geq\max\{\dim\g_{\alpha},\dim\g_{\beta}\},$$ with equality if and only if $\min\{\dim\g_{\alpha},\dim\g_{\beta}\}=1$.
\end{corintro}

\medskip

As an application of Theorem~\ref{thmintro:main}, we describe, in the second part of this paper, the graded subalgebras $\LLL$ of $\g(A)$ all whose elements are $\ad$-locally finite on $\LLL$; in particular, we obtain structure results for the solvable and nilpotent graded subalgebras of $\g(A)$. We recall that an element $x\in\LLL$ is \emph{$\ad$-locally finite} on $\LLL$ if for every $y\in\LLL$ there is some finite-dimensional subspace $V\subseteq \LLL$ containing $y$ such that $[x,V]\subseteq V$. In other words, $x\in\LLL$ is $\ad$-locally finite on $\LLL$ if and only if the exponential
$$\exp\ad x:=\sum_{n\in\NN}\frac{(\ad x)^n}{n!}$$
yields a well-defined operator in $\Aut(\LLL)$.
The condition that $\LLL$ only consists of such elements thus precisely means that $\LLL$ can be integrated to a group
$$G(\LLL)=\langle\exp\ad x \ | \ x\in\LLL\rangle\subseteq\Aut(\LLL).$$

Note that $\ad$-local finiteness is another key difference between real and imaginary root spaces: while $x$ is $\ad$-locally finite on $\g(A)$ for every $x\in\g_{\alpha}$ with $\alpha\in\Delta^{re}$ (this condition in fact characterises Kac--Moody algebras within the class of \emph{contragredient} Lie algebras, see \cite[\S 4.1]{MoPi}), the nonzero elements of imaginary root spaces are \emph{not} $\ad$-locally finite on $\g(A)$. We first establish a very precise form of this statement.

\begin{thmintro}
Let $\alpha\in\Delta^{im+}$ and $\beta\in\Delta^+$. Let $x\in\g_{\alpha}$ and $y\in\g_{\beta}$ be such that $[x,y]\neq 0$. Then $(\ad x)^ny\neq 0$ for all $n\in\NN$.
\end{thmintro}

We next state the announced structure result for graded subalgebras of $\LLL$ with only $\ad$-locally finite elements. An element $x\in\g(A)$ is \emph{homogeneous} if $x\in\g_{\alpha}$ for some $\alpha\in\Delta\cup\{0\}$. A set $\Psi\subseteq\Delta$ of roots is called \emph{closed} if $\alpha+\beta\in\Psi$ whenever $\alpha,\beta\in\Psi$ and $\alpha+\beta\in\Delta$. One then writes $\g_{\Psi}:=\bigoplus_{\alpha\in\Psi}\g_{\alpha}\subseteq\g(A)$. In particular, $\nn^{im\pm}=\g_{\Delta^{im\pm}}$, where $\Delta^{im\pm}=\Delta^{im}\cap\Delta^{\pm}$.

\begin{thmintro}\label{thmintro:structure_thm}
Let $\LLL$ be a graded subalgebra of $\g(A)$ such that each homogeneous element of $\LLL$ is $\ad$-locally finite on $\LLL$. Then there exists a closed set of real roots $\Psi\subseteq\Delta^{re}$, and abelian subalgebras $\LLL_0\subseteq\hh$, $\LLL^{im+}\subseteq\nn^{im+}$ and $\LLL^{im-}\subseteq\nn^{im-}$ such that 
\begin{enumerate}
\item
$\LLL=\LLL_0\oplus\g_{\Psi}\oplus\LLL^{im+}\oplus\LLL^{im-}$;
\item
$[\g_{\Psi},\LLL^{im+}]=\{0\}=[\g_{\Psi},\LLL^{im-}]$;
\item
$[\LLL^{im+},\LLL^{im-}]\subseteq \LLL_0\oplus\g_{\Psi}$.
\end{enumerate}
\end{thmintro}

Note that the subspaces $\g_{\Psi}$ with $\Psi\subseteq\Delta^{re}$ a closed set of real roots were completely described in \cite{CM18} (see Proposition~\ref{prop:closed_set_real_roots} below). On the other hand, we also provide a complete description of the graded abelian subalgebras $\LLL^{im\pm}$ of $\nn^{im\pm}$ (see Proposition~\ref{prop:abelian_subalgebra_nim+}).

As every element $x$ of a nilpotent subalgebra $\LLL$ of $\g(A)$ is \emph{$\ad$-locally nilpotent} on $\LLL$ (i.e. for every $y\in \LLL$ there exists some $n\in\NN$ such that $(\ad x)^ny=0$), Theorem~\ref{thmintro:structure_thm} applies in particular to nilpotent graded subalgebras of $\g(A)$, and yields the following analogue in the Kac--Moody setting of a classical result from the theory of finite-dimensional nilpotent Lie algebras.

\begin{corintro}
Let $\LLL$ be a graded subalgebra of $\g(A)$. Then $\LLL$ is nilpotent if and only if every homogeneous $x\in\LLL$ is $\ad$-locally nilpotent on $\LLL$. 
\end{corintro}

Together with \cite{CM18}, Theorem~\ref{thmintro:structure_thm} further implies the existence of a uniform bound on the nilpotency class of nilpotent graded subalgebras of $\g(A)$ (see Remark~\ref{remark:nilpotency_class} for more details on the bound).

\begin{corintro}
There exists some $N\in\NN$ depending only on $A$, such that every nilpotent graded subalgebra of $\g(A)$ has nilpotency class at most $N$.
\end{corintro}

Finally, we obtain an analogue in the Kac--Moody setting of another classical result, this time from the theory of finite-dimensional solvable Lie algebras. Let $\Gamma(A)$ denote the \emph{Dynkin diagram} of $A$.

\begin{thmintro}\label{thmintro:solvable}
Assume that $\Gamma(A)$ does not contain any subdiagram of affine type. Let $\LLL$ be a graded subalgebra of $\g(A)$.
Then the following assertions are equivalent:
\begin{enumerate}
\item
$\LLL$ is solvable.
\item
$\LLL^1:=[\LLL,\LLL]$ is nilpotent.
\item
$[\hh\cap\LLL^1,\LLL]=\{0\}$ and each homogeneous element of $\LLL$ is $\ad$-locally finite on $\LLL$.
\end{enumerate}
\end{thmintro}

Note that when $\Gamma(A)$ contains a subdiagram of affine type, Theorem~\ref{thmintro:solvable} does not hold anymore (see Example~\ref{example:solvable_locfin}). Nevertheless, a weaker form of Theorem~\ref{thmintro:solvable} can still be obtained without this additional hypothesis (see Theorem~\ref{thm:solvable_general}).

As a last remark, note that \emph{$\ad$-locally finite} subalgebras $\LLL$ of $\g(A)$ (in the sense that for every $y\in\g(A)$, there is some finite-dimensional subspace $V\subseteq\g(A)$ containing $y$ such that $\ad(\LLL)V\subseteq V$) were described in \cite[Theorem~3]{KP87}. Although related, this local finiteness condition is far more restrictive than only requiring the elements of $\LLL$ to be \emph{individually} $\ad$-locally finite \emph{on $\LLL$} (for instance, it forces $\LLL$ to be finite-dimensional), and in particular does not allow to describe the solvable and nilpotent (graded) subalgebras of $\g(A)$.

\subsection*{Conventions}
Throughout this paper, $\NN$ denotes the set of nonnegative integers, $\NN^*$ the set of positive integers, and $\ZZ^*$ the set of nonzero integers.

\section{Preliminaries}\label{section:preliminaries}
In this section, we fix some terminology and recall some basic facts about Kac--Moody algebras. The general reference for this section is \cite[Chapters~1--5 and \S 9.11]{Kac}. 

\subsection{Generalised Cartan matrices}
A {\bf generalised Cartan matrix} (GCM) is an integral matrix $A=(a_{ij})_{i,j\in I}$ indexed by some finite set $I$ such that
\begin{enumerate}
\item[(C1)] $a_{ii}=2$ for all $i\in I$;
\item[(C2)] $a_{ij}\leq 0$ for all $i,j\in I$ with $i\neq j$;
\item[(C3)] $a_{ij}=0\iff a_{ji}=0$ for all $i,j\in I$.
\end{enumerate}
The matrix $A$ is called {\bf symmetrisable} if there exists some diagonal matrix $D$ and some symmetric matrix $B$ such that $A=DB$.

\subsection{Kac--Moody algebras}\label{subsection:KMA}
Let $A=(a_{ij})_{i,j\in I}$ be a symmetrisable generalised Cartan matrix, and fix a realisation $(\hh,\Pi=\{\alpha_i \ | \ i \in I\},\Pi^{\vee}=\{\alpha^{\vee}_i \ | \ i \in I\})$ of $A$ as in \cite[Chapter~1]{Kac}. The {\bf Kac--Moody algebra} $\g(A)$ is the complex Lie algebra with generators $e_i,f_i$ ($i\in I$) and $\hh$, and defining relations
\begin{align}
[h,h']&=0 \quad\textrm{for all $h,h'\in\hh$;}\\
[h,e_i]&=\la\alpha_i,h\ra e_i\quad\textrm{and}\quad [h,f_i]=-\la\alpha_i,h\ra f_i \quad\textrm{for all $i\in I$;}\\
[f_j,e_i]&=\delta_{ij}\alpha_i^{\vee} \quad\textrm{for all $i,j\in I$;}\\
(\ad e_i)^{1-a_{ij}}e_j&=0\quad\textrm{and}\quad (\ad f_i)^{1-a_{ij}}f_j=0\quad\textrm{for all $i,j\in I$ with $i\neq j$.}
\end{align}
The elements $e_i,f_i$ ($i\in I$), as well as the space $\hh$, are identified with their canonical image in $\g(A)$, and are respectively called the {\bf Chevalley generators} and {\bf Cartan subalgebra} of $\g(A)$. The subalgebra of $\g(A)$ generated by the $e_i$ (resp. $f_i$) for $i\in I$ is denoted $\nn^+=\nn^+(A)$ (resp. $\nn^-$), and $\g(A)$ admits a triangular decomposition $$\g(A)=\nn^-\oplus\hh\oplus\nn^+\quad\textrm{(direct sum of vector spaces)}.$$
The adjoint action of $\hh$ on $\g(A)$ is diagonalisable, yielding a {\bf root space decomposition} 
$$\g(A)=\hh\oplus\bigoplus_{\alpha\in\Delta}\g_{\alpha},$$
where $\g_{\alpha}:=\{x\in\g(A) \ | \ [h,x]=\alpha(h)x \ \forall h\in\hh\}$ is the {\bf root space} attached to $\alpha\in\hh^*$, and where $\Delta:=\{\alpha\in\hh^*\setminus\{0\} \ | \ \g_{\alpha}\neq\{0\}\}$ is the corresponding set of {\bf roots}. 

Set $Q:=\bigoplus_{i\in I}\ZZ\alpha_i$. Every root $\alpha\in\Delta$ either belongs to $$Q_+:=\bigoplus_{i\in I}\NN\alpha_i$$ (in which case $\alpha$ is called {\bf positive}) or to $Q_-:=-Q_+$ (in which case $\alpha$ is called {\bf negative}); writing $\alpha=\sum_{i\in I}n_i\alpha_i$ for some $n_i\in\ZZ$, the number $\height(\alpha):=\sum_{i\in I}n_i$ is called the {\bf height} of $\alpha$. The set of all positive (resp. negative) roots is denoted $\Delta^+$ (resp. $\Delta^-$). 
An element $x\in\g_{\alpha}$ for some $\alpha\in\Delta\cup\{0\}$ (where $\g_0:=\hh$) is called {\bf homogeneous} of {\bf degree} $\deg(x):=\alpha$.

Setting $\hh':=\sum_{i\in I}\CC\alpha_i^{\vee}\subseteq\hh$, the derived algebra $\g'(A):=[\g(A),\g(A)]$ of $\g(A)$ has a triangular decomposition
$$\g'(A)=\nn^-\oplus\hh'\oplus\nn^+.$$
It has the same presentation as $\g(A)$, with $\hh$ replaced by $\hh'$. The center $\mathfrak c$ of $\g'(A)$ is contained in $\hh'$, and $\g'(A)/\mathfrak{c}$ is a simple Lie algebra. Moreover, if $\widetilde{\nn}^+$ denotes the free Lie algebra with generators $e_i$ ($i\in I$) and $\mathfrak{i}^+$ the ideal of $\widetilde{\nn}^+$ generated by the elements $(\ad e_i)^{1-a_{ij}}e_j$ ($i\neq j$), then the assignment $e_i\mapsto e_i$ defines an isomorphism
\begin{equation}\label{eqn:presentation_nn+}
\nn^+\cong \widetilde{\nn}^+/\mathfrak{i}^+.
\end{equation}

To any subset $\Psi\subseteq\Delta$, we associate the subspace
$$\g_{\Psi}:=\bigoplus_{\alpha\in\Psi}\g_{\alpha}$$
of $\g(A)$. The set $\Psi$ is {\bf closed} if $\alpha+\beta\in\Psi$ whenever $\alpha,\beta\in\Psi$ and $\alpha+\beta\in\Delta$.

The assignment
$$\omega(e_i):=-f_i, \quad \omega(f_i):=-e_i, \quad\textrm{and}\quad \omega(h):=-h\quad\textrm{for all $i\in I$ and $h\in\hh$}$$
defines an involutive automorphism $\omega$ of $\g(A)$, called the {\bf Chevalley involution}. Note that $\omega(\g_{\alpha})=\g_{-\alpha}$ for all $\alpha\in\Delta$; in particular, $\Delta^-=-\Delta^+$.

\subsection{Weyl group of $\g(A)$}\label{subsection:WGOgA}
The {\bf Weyl group} $\WW=\WW(A)$ of $\g(A)$ is the subgroup of $\GL(\hh^*)$ generated by the {\bf simple reflections}
$$s_i\co\hh^*\to\hh^*, \quad \alpha\mapsto \alpha-\la\alpha,\alpha_i^{\vee}\ra \alpha_i$$
for $i\in I$; the couple $(\WW,\{s_i \ | \ i\in I\})$ is then a Coxeter system. Alternatively, $\WW$ can be identified with the subgroup of $\GL(\hh)$ generated by the ``dual'' simple reflections
$$s_i^{\vee}\co\hh\to\hh, \quad h\mapsto h-\la\alpha_i,h\ra \alpha_i^{\vee}.$$

For each $i\in I$, the element $s_i^*:=\exp(\ad f_i)\exp(\ad e_i)\exp(\ad f_i)$ defines an automorphism of $\g(A)$, and the assignment $s_i^*\mapsto s_i$ defines a surjective group morphism $\pi\co\WW^*\to\WW$ from the group
$$\WW^*:=\la s_i^* \ | \ i\in I\ra\subseteq\Aut(\g(A))$$
to $\WW$. Moreover, the restriction of $\WW^*$ to $\hh$ coincides with $\WW\subseteq\GL(\hh)$, and
\begin{equation}
w^*\g_{\alpha}=\g_{w\alpha}\quad\textrm{for all $\alpha\in\Delta\cup\{0\}$ and all $w^*\in\WW^*$ with $\pi(w^*)=w$.}
\end{equation}
In particular, $\WW$ stabilises $\Delta\subseteq\hh^*$.

A root $\alpha\in\Delta$ is called {\bf real} if it belongs to $\Delta^{re}:=\WW\cdot\Pi$; otherwise, $\alpha$ is called {\bf imaginary}, and we set $\Delta^{im}:=\Delta\setminus\Delta^{re}$. We further set
$$\Delta^{re\pm}:=\Delta^{re}\cap\Delta^{\pm}\quad\textrm{and}\quad\Delta^{im\pm}:=\Delta^{im}\cap\Delta^{\pm}.$$
Then $\Delta^{im\pm}$ is a closed set of roots stabilised by $\WW$. In particular,
$$\nn^{im\pm}:=\g_{\Delta^{im\pm}}$$
is a $\WW^*$-invariant subalgebra of $\nn^{\pm}$.

If $\alpha=w\alpha_i$ for some $w\in\WW$ and $i\in I$, then $\alpha^{\vee}:=w\alpha_i^{\vee}$ depends only on $\alpha$, and is called the {\bf coroot} associated to $\alpha$. For each $\alpha\in\Delta^{re+}$, we fix a decomposition $\alpha=w_{\alpha}\alpha_i$ for some $w_{\alpha}\in\WW$ and $i\in I$ (with $w_{\alpha}:=1$ if $\alpha\in\Pi$). We also choose some $w_{\alpha}^*\in\WW^*$ with $\pi(w_{\alpha}^*)=w_{\alpha}$, and we set $$e_{\alpha}:=w_{\alpha}^*e_i\quad\textrm{and}\quad e_{-\alpha}:=w_{\alpha}^*f_i$$ (the element $e_{\alpha}$ is in fact independent of the choices of $i,w_{\alpha},w_{\alpha}^*$ up to a sign, but we will not need this fact). Thus  $\g_{\alpha}=\CC e_{\alpha}$, and
\begin{equation}\label{eqn:sl2_alpha}
[e_{-\alpha},e_{\alpha}]=\alpha^{\vee}\quad\textrm{and}\quad [\alpha^{\vee},e_{\pm\alpha}]=\pm 2e_{\pm\alpha}\quad\textrm{for all $\alpha\in\Delta^{re}$}.
\end{equation}

For any $\alpha\in\Delta$, we have
\begin{equation}\label{eqn:multiple_roots}
\ZZ\alpha\cap\Delta=\{\pm\alpha\}\quad\textrm{if $\alpha\in\Delta^{re}$}\quad\textrm{and}\quad \ZZ\alpha\cap\Delta=\ZZ^*\alpha\quad\textrm{if $\alpha\in\Delta^{im}$.}
\end{equation}

\subsection{Coxeter diagram of $A$}
The {\bf Coxeter diagram} $\Gamma(A)$ of $A$ is the graph with vertex set $\Pi$ and with an edge between $\alpha_i$ and $\alpha_j$ if and only if $a_{ij}<0$. We call $\Gamma(A)$ of {\bf affine type} if the corresponding Dynkin diagram is of affine type, in the sense of \cite[\S 4.8]{Kac}. The {\bf support} of an element $\alpha=\sum_{i\in I}n_i\alpha_i\in Q_+$ is the subdiagram $\supp(\alpha)$ of $\Gamma(A)$ with vertex set $\{\alpha_i \ | \ n_i\neq 0\}$. Here, by \emph{subdiagram of $\Gamma(A)$ with vertex set $S\subseteq\Pi$}, we always mean the subgraph of $\Gamma(A)$ with vertex set $S$ and with all the edges connecting the vertices in $S$. When convenient (and when no confusion is possible), we will also view $\supp(\alpha)$ as the subset $J$ of $I$ such that $\supp(\alpha)$ has vertex set $\{\alpha_j \ | \ j\in J\}$.

 Similarly, for any $w\in\WW$ with reduced decomposition $w=s_{i_1}\dots s_{i_d}$, the set $I_w:=\{i_1,\dots,i_d\}\subseteq I$ depends only on $w$, and we call the subdiagram $\supp(w)$ of $\Gamma(A)$ with vertex set $\{\alpha_i \ | \ i\in I_w\}$ the {\bf support} of $w$. 
 
 Any root $\alpha\in\Delta$ has connected support. Moreover, setting
$$K_0:=\{\alpha\in Q_+ \ | \ \textrm{$\la\alpha,\alpha_i^{\vee}\ra\leq 0$ for all $i\in I$} \},$$
the set $\Delta^{im+}$ of positive imaginary roots can be described as
\begin{equation}\label{eqn:Delta_im+}
\Delta^{im+}=\WW\cdot \{\alpha\in K_0 \ | \ \textrm{$\supp(\alpha)$ is connected}\}.
\end{equation}
Note that the second statement in (\ref{eqn:multiple_roots}) follows from (\ref{eqn:Delta_im+}).

\subsection{Invariant bilinear form of $\g(A)$}
Since $A$ is symmetrisable, $\g(A)$ admits a symmetric {\bf invariant bilinear form} $(\cdot|\cdot)\co\g(A)\times\g(A)\to\CC$ (see \cite[\S 2.3]{Kac}). The restriction of $(\cdot|\cdot)$ to $\hh$ is nondegenerate, and we denote by $\hh^*\to\hh, \ \alpha\mapsto\alpha^{\sharp}$ the induced isomorphism, characterised by
$$\la\beta,\alpha^{\sharp}\ra=(\beta^{\sharp}|\alpha^{\sharp})=:(\beta|\alpha)\quad\textrm{for all $\alpha,\beta\in\hh^*$}.$$ 
Then $(\alpha_i|\alpha_j)\leq 0$ for all $i,j\in I$ with $i\neq j$, and
\begin{equation}\label{eqn:coroot}
\alpha^{\vee}=\frac{2\alpha^{\sharp}}{(\alpha|\alpha)}\quad\textrm{for all $\alpha\in\Delta^{re}$.}
\end{equation}
Note also that 
$$[\alpha^{\sharp},x_{\beta}]=(\beta|\alpha)x_{\beta}\quad\textrm{for all $\alpha,\beta\in\hh^*$ and $x_{\beta}\in\g_{\beta}$}$$
and that 
\begin{equation}\label{eqn:prop_inv_form}
[\g_{-\alpha},\g_{\alpha}]=\CC\alpha^{\sharp}\quad\textrm{for all $\alpha\in\Delta$.}
\end{equation}
The symmetric bilinear form $$\hh^*\times\hh^*\to\CC, \ (\alpha,\beta)\mapsto (\alpha|\beta)$$ is $\WW$-invariant, and we have
\begin{equation}\label{eqn:norm_roots}
(\alpha|\alpha)>0\quad\textrm{for all $\alpha\in\Delta^{re}$}\quad\textrm{and}\quad (\beta|\beta)\leq 0 \quad\textrm{for all $\beta\in\Delta^{im}$}.
\end{equation}
A root $\alpha\in\Delta$ is {\bf isotropic} if $(\alpha|\alpha)=0$; otherwise, it is {\bf non-isotropic}. We denote by $\Delta^{im+}_{is}$ (resp. $\Delta^{im+}_{an}$) the set of isotropic (resp. non-isotropic) positive imaginary roots. 
If $\alpha\in K_0$, then
\begin{equation}\label{eqn:isotropic}
(\alpha|\alpha)=0 \iff \textrm{$\supp(\alpha)$ is a subdiagram of affine type}
\end{equation}
and
\begin{equation}\label{eqn:isotropic_zero}
(\alpha|\alpha)=0 \implies (\alpha|\alpha_i)=0 \quad \textrm{whenever $\alpha_i\in\supp(\alpha)$}.
\end{equation}
Moreover, if $\supp(\alpha)$ is of affine type and $\beta\in\Delta^{im}$, then
\begin{equation}\label{eqn:isotropic_multiple}
\supp(\beta)\subseteq\supp(\alpha) \implies \beta\in\QQ\alpha.
\end{equation}

\subsection{Closed sets of real roots}
Finally, we record for future reference the following result from \cite{CM18}. 

\begin{prop}\label{prop:closed_set_real_roots}
Let $\Psi\subseteq\Delta^{re}$ be a closed set of real roots and let $\g$ be the subalgebra of $\g(A)$ generated by $\g_{\Psi}$. Set $\Psi_s:=\{\alpha\in\Psi \ | \ -\alpha\in\Psi\}$ and $\Psi_n:=\Psi\setminus\Psi_s$. Set also $\hh_s:=\sum_{\gamma\in\Psi_s}\CC\gamma^{\vee}$, $\g_s:=\hh_s\oplus\g_{\Psi_s}$ and $\g_n:=\g_{\Psi_n}$. Then
\begin{enumerate}
\item
$\g_s$ is a subalgebra and $\g_n$ is an ideal of $\g$. In particular, $\g=\g_s\ltimes\g_n$.
\item
$\g_n$ is the largest nilpotent ideal of $\g$.
\item
$\g_s$ is a semisimple finite-dimensional Lie algebra with Cartan subalgebra $\hh_s$ and set of roots $\Psi_s$.
\end{enumerate}
Moreover, there exists some $N\in\NN$, depending only on the GCM $A$, such that $\g_n$ has nilpotency class at most $N$.
\end{prop}
\begin{proof}
The assertions (1)--(3) are contained in the main theorem of \cite{CM18}. Since, in the terminology of \cite{CM18}, $\Psi_n$ is pro-nilpotent by \cite[Proposition~7 and Lemma~8]{CM18}, the existence of $N\in\NN$ follows from (the proof of) \cite[Lemma~1]{CM18}.
\end{proof}

\section{Basic properties of roots}
In this section, we collect a few useful properties of roots and root spaces. We fix again a symmetrisable GCM $A$, and keep all notations from Section~\ref{section:preliminaries}.

We first recall the properties of root strings.
\begin{lemma}\label{lemma:RS}
Let $\alpha\in\Delta^{re}$ and $\beta\in\Delta$. Set $S(\alpha,\beta):=(\beta+\ZZ\alpha)\cap\Delta$. Then $S(\alpha,\beta)=\{\beta+n\alpha \ | \ -p\leq n\leq q\}$ for some $p,q\in\NN$ with $p-q=\la\beta,\alpha^{\vee}\ra$, and one of the following holds:
\begin{enumerate}
\item $S(\alpha,\beta)\cap\Delta^{re}=\varnothing$.
\item $|S(\alpha,\beta)\cap\Delta^{re}|=1$; in that case, $S(\alpha,\beta)=\{\beta\}$.
\item $|S(\alpha,\beta)\cap\Delta^{re}|=2$; in that case, $S(\alpha,\beta)\cap\Delta^{re}=\{\beta-p\alpha,\beta+q\alpha\}$.
\item $|S(\alpha,\beta)\cap\Delta^{re}|=3$; in that case, $S(\alpha,\beta)\subseteq\Delta^{re}$.
\item $|S(\alpha,\beta)\cap\Delta^{re}|=4$; in that case, $$S(\alpha,\beta)\cap\Delta^{re}=\{\beta-p\alpha, \ \beta-(p-1)\alpha, \ \beta+(q-1)\alpha, \ \beta+q\alpha\}.$$
\end{enumerate}
\end{lemma}
\begin{proof}
See \cite[Proposition~1]{BP95}.
\end{proof}

\begin{lemma}\label{lemma:csupp}
Let $\beta\in\Delta^{im+}$. Then the following assertions hold:
\begin{enumerate}
\item
There is a unique $\beta'\in \WW.\beta\cap K_0$ \emph{(}namely, the unique element of $\WW.\beta$ of minimal height\emph{)}.
\item
$\supp(\beta')$ is a subdiagram of $\supp(\beta)$.
\item
$\beta'=w\beta$ for some $w\in\WW$ with $\supp(w)\subseteq\supp(\beta)$.
\end{enumerate}
\end{lemma}
\begin{proof}
Note that $\beta\in K_0$ if and only if $\height(s_i\beta)\geq\height(\beta)$ for all $i\in I$. The first statement then follows from \cite[Proposition~5.2b]{Kac}. By uniqueness of $\beta'$, there is a sequence of elements $i_1,\dots,i_d\in I$ such that the roots $\beta_t:=s_{i_t}\dots s_{i_1}\beta$ ($t=0,\dots,d$) satisfy $\beta_0=\beta$, $\beta_d=\beta'$ and $\height(\beta_{t})<\height(\beta_{t-1})$ ($t=1,\dots,d$). In particular, $$\supp(\beta')=\supp(\beta_d)\subseteq \supp(\beta_{d-1})\subseteq\dots\subseteq \supp(\beta_0)=\supp(\beta)$$ and $\alpha_{i_t}\in\supp(\beta_{t-1})\subseteq\supp(\beta)$ for all $t=1,\dots,d$, yielding (2) and (3).
\end{proof}

Given $i\in I$, and $\alpha=\sum_{j\in I}n_j\alpha_j\in Q$, we set $\height_{\alpha_i}(\alpha):=n_i$.

\begin{lemma}\label{lemma:wK_0}
Let $\alpha\in\Delta^{im+}\cap K_0$ and $i\in I$ be such that $\la\alpha,\alpha_i^{\vee}\ra\neq 0$. Let $w\in\WW$ be such that $\height_{\alpha_i}(\alpha)=\height_{\alpha_i}(w\alpha)$. Then $\alpha_i\notin\supp(w)$.
\end{lemma}
\begin{proof}
Let $w=s_{i_1}\dots s_{i_d}$ be a reduced decomposition of $w$. For each $t\in\{1,\dots,d\}$, set $\beta_t:=s_{i_1}\dots s_{i_{t-1}}\alpha_{i_t}\in\Delta^{re+}$, so that $$\{\beta_t \ | \ 1\leq t\leq d\}=\Delta^+\cap w\Delta^-$$ (see e.g. \cite[Exercise~4.33]{KMGbook}). Then $$\alpha-w\alpha=\sum_{t=1}^d\la\alpha,\alpha_{i_t}^{\vee}\ra \beta_t$$ (see e.g. \cite[Exercise~4.34]{KMGbook}). Hence $$0=\height_{\alpha_i}(\alpha-w\alpha)=\sum_{t=1}^d\la\alpha,\alpha_{i_t}^{\vee}\ra \height_{\alpha_i}(\beta_t).$$ Since $\la\alpha,\alpha_{i_t}^{\vee}\ra\leq 0$ for all $t=1,\dots,d$, this implies that $$\la\alpha,\alpha_{i_t}^{\vee}\ra \height_{\alpha_i}(\beta_t)=0\quad\textrm{for all $t=1,\dots,d$.}$$ Assume for a contradiction that $\alpha_i\in\supp(w)$, and let $r\in\{1,\dots,d\}$ be minimal such that  $i_r=i$. Then $\height_{\alpha_i}(\beta_r)=1$ and hence  $\la\alpha,\alpha_i^{\vee}\ra= 0$, a contradiction.
\end{proof}

\begin{lemma}\label{lemma:imaginary_not_connected}
Let $\beta_1,\dots,\beta_r\in\Delta^{im+}$ be such that $\beta_i+\beta_j\notin\Delta$ for all $i\neq j$. Then there exists some $w\in\WW$ such that $w\beta_t\in K_0$ for all $t=1,\dots,r$, and such that $\supp(w\beta_1),\dots,\supp(w\beta_r)$ are $r$ distinct connected components in the subdiagram $\supp(w\beta_1)\cup\dots\cup\supp(w\beta_r)$.
\end{lemma}
\begin{proof}
We prove the claim by induction on $r$.
For $r=1$, the claim is clear. Let now $r=2$. Since $\WW\cdot\Delta^{im+}=\Delta^{im+}\subseteq Q_+$, we find some $w\in\WW$ such that $w(\beta_1+\beta_2)\in Q_+$ has minimal height in $\WW.(\beta_1+\beta_2)$. Thus, $w\beta_1+w\beta_2\in K_0\setminus\Delta^{im+}$, so that $\supp(w\beta_1)$ and $\supp(w\beta_2)$ are two distinct connected components of $\supp(w\beta_1)\cup\supp(w\beta_2)$ by (\ref{eqn:Delta_im+}). In particular, $w\beta_j\in K_0$ for $j=1,2$, yielding the claim in that case: otherwise, we find by Lemma~\ref{lemma:csupp} some $v_j\in\WW$ ($j=1,2$) such that $\supp(v_j)\subseteq\supp(w\beta_j)$ and $\height(v_1w\beta_1)+\height(v_2w\beta_2)<\height(w\beta_1)+\height(w\beta_2)$ (in particular, $v_1v_2=v_2v_1$ and $v_iw\beta_{j}=w\beta_j$ for $i\neq j$), contradicting the fact that
$$\height(v_1w\beta_1)+\height(v_2w\beta_2)=\height(v_1v_2w(\beta_1+\beta_2)).$$

Assume next that the claim holds for some $r\geq 2$, and let us prove it for $r+1$. By induction hypothesis, there is no loss of generality in assuming that $\beta_t\in K_0$ for all $t=1,\dots,r$, and that $\supp(\beta_1),\dots,\supp(\beta_{r})$ are $r$ distinct connected components in the subdiagram $\supp(\beta_1)\cup\dots\cup\supp(\beta_r)$. 

We claim that $\supp(\beta_t)\cup\supp(\beta_{r+1})$ is not connected for any given $t\in\{1,\dots,r\}$. Indeed, by the case $r=2$, we find some $v_t\in\WW$ such that $v_t\beta_t\in K_0$ and $v_t\beta_{r+1}\in K_0$, and such that $\supp(v_t\beta_t)$ and $\supp(v_t\beta_{r+1})$ are distinct connected components of $\supp(v_t\beta_t)\cup\supp(v_t\beta_{r+1})$. Note that $v_t\beta_t=\beta_t$ by Lemma~\ref{lemma:csupp}(1) and $\supp(v_t\beta_{r+1})\subseteq\supp(\beta_{r+1})$ by Lemma~\ref{lemma:csupp}(2). Hence, if $\supp(\beta_t)\cup\supp(\beta_{r+1})$ were connected, there would exist some $i\in I$ with $\alpha_i\in \supp(\beta_{r+1})\setminus\supp(v_t\beta_{r+1})$ such that $\alpha_i\notin\supp(\beta_t)$ and $\supp(\beta_t)\cup\{\alpha_i\}$ is connected. But then $\la\beta_t,\alpha_i^{\vee}\ra<0$, so that Lemma~\ref{lemma:wK_0} (with $\alpha:=\beta_t$ and $w:=v_t$) would imply that $\alpha_i\notin\supp(v_t)$, and hence that $\height_{\alpha_i}(\beta_{r+1})=\height_{\alpha_i}(v_t\beta_{r+1})=0$, a contradiction. 

Finally, Lemma~\ref{lemma:csupp} yields some $w\in\WW$ with $\supp(w)\subseteq\supp(\beta_{r+1})$ such that $w\beta_{r+1}\in K_0$. As $w\beta_t=\beta_t$ for all $t=1,\dots,r$, this completes the induction step.
\end{proof}

\begin{remark}
Note that, up to now, we did not make use of the symmetrisability assumption on $A$. In particular, Lemmas~\ref{lemma:csupp}, \ref{lemma:wK_0} and \ref{lemma:imaginary_not_connected} remain valid for an arbitrary GCM $A$.
\end{remark}

\begin{lemma}\label{lemma:basic_propI}
Let $\alpha,\beta\in\Delta^{im+}$. Then the following assertions hold:
\begin{enumerate}
\item
$(\alpha|\beta)\leq 0$.
\item
If $\alpha+\beta\in\Delta^{+}$, then $(\alpha|\beta)<0$ unless $\alpha,\beta$ are proportional isotropic roots.
\item
If $(\alpha|\beta)<0$ then $\alpha+\beta\in\Delta^{im+}$.
\end{enumerate}
\end{lemma}
\begin{proof}
The lemma sums up Exercises~5.16, 5.17, and 5.18 in \cite{Kac}. For the convenience of the reader, here are more detailed proofs.

(1) Using the $\WW$-action, there is no loss of generality in assuming that $\alpha\in K_0$ (see (\ref{eqn:Delta_im+})). But then $(\alpha|\alpha_i)\leq 0$ for all $i\in I$. Since $\beta\in Q_+$, the claim follows.

(2) Assume that $\alpha+\beta\in\Delta^{im+}$ and that $(\alpha|\beta)=0$. As in the proof of (1), there is no loss of generality in assuming that $\alpha\in K_0$, so that $(\alpha|\alpha_i)=0$ for all $i\in I$ with $\alpha_i\in\supp(\beta)$. In particular, $\supp(\beta)\subseteq\supp(\alpha)$: otherwise, since $\supp(\alpha)\cup\supp(\beta)=\supp(\alpha+\beta)$ is connected, there would exist some $j\in\supp(\beta)\setminus\supp(\alpha)$ such that the subdiagram $\supp(\alpha)\cup\{\alpha_j\}$ is connected, and hence $(\alpha|\alpha_j)<0$, a contradiction. Up to conjugating $\beta$ by some $w\in\WW$ with $\supp(w)\subseteq\supp(\beta)$ (so that $w\alpha=\alpha$), we may then assume by Lemma~\ref{lemma:csupp} that $\beta\in K_0$ as well. Exchanging the roles of $\alpha$ and $\beta$ in the above argument, we deduce that $\supp(\alpha)=\supp(\beta)$ and that $(\alpha|\alpha)=(\beta|\beta)=0$. Moreover, (\ref{eqn:isotropic}) implies that $\supp(\alpha)$ is a subdiagram of affine type, and hence $\alpha,\beta$ are proportional isotropic roots by (\ref{eqn:isotropic_multiple}), as desired.

(3) Up to replacing $\alpha+\beta$ with an element of minimal height in $\WW.(\alpha+\beta)\subseteq Q_+$, we may assume that $\alpha+\beta\in K_0$. We claim that $\supp(\alpha+\beta)$ is connected, so that (3) follows from (\ref{eqn:Delta_im+}). Otherwise, since $\supp(\alpha)$ and $\supp(\beta)$ are connected, they are distinct connected components of $\supp(\alpha+\beta)=\supp(\alpha)\cup\supp(\beta)$, and hence $(\alpha|\beta)=0$, a contradiction.
\end{proof}

\begin{lemma}\label{lemma:prop36}
Let $\alpha\in\Delta^{re}$ and $\beta\in\Delta$.
\begin{enumerate}
\item
 If $(\alpha|\beta)<0$, then $[x,y]\neq 0$ for all nonzero $x\in\g_{\alpha}$ and $y\in\g_{\beta}$.
 \item
 $[\g_{\alpha},\g_{\beta}]= \{0\}$ if and only if $\alpha+\beta\notin\Delta$. 
\end{enumerate}
\end{lemma}
\begin{proof}
(1) follows from \cite[Proposition 3.6(b,ii)]{Kac}. For (2), consider the adjoint action of $\g_{(\alpha)}:=\g_{-\alpha}+\CC\alpha^{\vee}+\g_{\alpha}\cong\sll_2(\CC)$ on $M:=\bigoplus_{n\in\ZZ}\g_{\beta+n\alpha}$. By \cite[Proposition 3.6]{Kac}, $M$ is finite-dimensional and decomposes as a direct sum of irreducible (graded) $\g_{(\alpha)}$-submodules. The $\sll_2(\CC)$-module theory (see \cite[Lemma~3.2]{Kac}) then implies that if $\beta+n\alpha\in\Delta$, then $[e_{\alpha},\g_{\beta+n\alpha}]=\{0\}$ if and only if $\beta+(n+1)\alpha\notin\Delta$. The case $n=0$ now yields the claim.
\end{proof}

\begin{lemma}\label{lemma:rerere}
Let $\alpha,\beta\in\Delta^{re}$ be such that $\alpha+\beta\in\Delta^{re}$. Then $[\g_{\alpha},\g_{\beta}]=\g_{\alpha+\beta}$.
\end{lemma}
\begin{proof}
This follows from Lemma~\ref{lemma:prop36}(2).
\end{proof}

\begin{lemma}\label{lemma:dualbases}
Let $\alpha\in\Delta$, and let $x\in\g_{\alpha}$ be nonzero. Then there is some nonzero $x'\in\CC x$ such that $[\omega(x'),x']=\alpha^{\sharp}$. 
\end{lemma}
\begin{proof}
Since $(\omega(x)|x)\neq 0$ by \cite[Theorem~11.7a)]{Kac}, the claim follows from \cite[Theorem~2.2e)]{Kac}.
\end{proof}

\begin{lemma}\label{lemma:Kac_free_Lie}
Let $\alpha\in\Delta^{im+}$.
\begin{enumerate}
\item
If $(\alpha|\alpha)<0$, then $\g_{\NN^*\alpha}=\bigoplus_{n\in\NN^*}\g_{n\alpha}$ is a free Lie algebra on a basis of the space $\bigoplus_{n\in\NN^*}\g^0_{n\alpha}$, where $\g^0_{n\alpha}:=\{x\in\g_{n\alpha} \ | \ \textrm{$(x|\omega(y))=0$ for all $y\in [\g_{\NN^*\alpha},\g_{\NN^*\alpha}]$}\}$.
\item
If $(\alpha|\alpha)=0$, then $\CC\alpha^{\sharp}\oplus\bigoplus_{n\in\ZZ^*}\g_{n\alpha}$ is an infinite Heisenberg Lie algebra. In particular, $[\g_{n\alpha},\g_{-m\alpha}]=\CC \delta_{m,n}\alpha^{\sharp}$ for all $m,n\in\ZZ^*$.
\end{enumerate}
\end{lemma}
\begin{proof}
This follows from \cite[Corollary~9.12]{Kac}.
\end{proof}

\section{Structure of $\nn^{im+}$}
This section is devoted to the proofs of Theorem~\ref{thmintro:main} and Corollaries~\ref{corintro:main} and \ref{corintro:dimension}. We fix again a symmetrisable GCM $A$, and keep all notations from \S\ref{section:preliminaries}.

The proof of the following lemma is essentially the same as the proof of \cite[Proposition~9.12]{Kac}; for the benefit of the reader, we repeat here the argument.
\begin{lemma}\label{lemma:basic1}
Let $\alpha,\beta\in\Delta^{im+}$ be such that $(\alpha|\beta)<0$. Let $x_{\alpha}\in\g_{\alpha}$ and $y_{\beta}\in\g_{\beta}$ be nonzero and such that $[\omega(y_{\beta}),x_{\alpha}]=0$. Then $x_{\alpha},y_{\beta}$ generate a free Lie algebra.
\end{lemma}
\begin{proof}
Set $x_{\alpha}^+:=x_{\alpha}$, $x_{\alpha}^-:=\omega(x_{\alpha})\in\g_{-\alpha}$, $y_{\beta}^+:=y_{\beta}$ and $y_{\beta}^-:=\omega(y_{\beta})\in\g_{-\beta}$. By Lemma~\ref{lemma:dualbases}, up to replacing $x_{\alpha}$ (resp. $y_{\beta}$) by a nonzero multiple, we may assume that 
$$[x_{\alpha}^-,x_{\alpha}^+]=\alpha^{\sharp}\quad\textrm{and}\quad [y_{\beta}^-,y_{\beta}^+]=\beta^{\sharp}.$$
Hence
$$[\alpha^{\sharp},x^{\pm}_{\alpha}]=\pm(\alpha|\alpha)x^{\pm}_{\alpha}, \quad [\alpha^{\sharp},y^{\pm}_{\beta}]=\pm(\alpha|\beta)y^{\pm}_{\beta}, $$
and
$$[\beta^{\sharp},x^{\pm}_{\alpha}]=\pm(\alpha|\beta)x^{\pm}_{\alpha},\quad [\beta^{\sharp},y^{\pm}_{\beta}]=\pm(\beta|\beta)y^{\pm}_{\beta}.$$
Moreover,
$$[y_{\beta}^{+},x_{\alpha}^{-}]=0=[y_{\beta}^{-},x_{\alpha}^{+}]$$
by assumption. 

Set $B:=\begin{psmallmatrix}(\alpha|\alpha) & (\alpha|\beta)\\ (\alpha|\beta) & (\beta|\beta)\end{psmallmatrix}$, let $(\hh_{B},\{\gamma_1,\gamma_2\},\{\gamma_1^{\vee},\gamma_2^{\vee}\})$ be a realisation of $B$, and let $\widetilde{\g}(B)$ be the corresponding Lie algebra defined in \cite[\S 1.2]{Kac}, with Chevalley generators $e_1^B,e_2^B$ and $f_1^B,f_2^B$. Then $\widetilde{\g}(B)$ has a triangular decomposition $$\widetilde{\g}(B)=\widetilde{\nn}^-_B\oplus \hh_B \oplus \widetilde{\nn}^+_B,$$ where $\widetilde{\nn}^+_B$ (resp. $\widetilde{\nn}^-_B$) is freely generated by $e_1^B,e_2^B$ (resp. $f_1^B,f_2^B$), see \cite[Theorem~1.2]{Kac}. Let also $\widetilde{\g}'(B)=\widetilde{\nn}^-_B\oplus \hh'_B\oplus \widetilde{\nn}^+_B$ denote the derived subalgebra of $\widetilde{\g}(B)$, where $\hh'_B:=\CC\gamma_1^{\vee}\oplus\CC\gamma_2^{\vee}\subseteq\hh_B$.

We have just shown that the assignment
$$e_1^B\mapsto x^+_{\alpha}, \quad f_1^B\mapsto x^-_{\alpha}, \quad e_2^B\mapsto y^+_{\beta}, \quad f_2^B\mapsto y^-_{\beta}$$
defines a surjective Lie algebra morphism $\phi\co\widetilde{\g}'(B)\to \LLL$ from $\widetilde{\g}'(B)$ to the subalgebra $\LLL$ of $\g(A)$ generated by $x^{\pm}_{\alpha},y^{\pm}_{\beta}$ (see \cite[Remark~1.5]{Kac}). On the other hand, if $\rho\in\hh_B^*$ is chosen so that $$(\rho|\gamma_1)_B=\tfrac{1}{2}(\gamma_1|\gamma_1)_B=\tfrac{1}{2}(\alpha|\alpha)\quad\textrm{and}\quad (\rho|\gamma_2)_B=\tfrac{1}{2}(\gamma_2|\gamma_2)_B=\tfrac{1}{2}(\beta|\beta),$$ where $(\cdot|\cdot)_B$ denotes the bilinear form on $\hh_B^*$ induced by the invariant bilinear form on $\g(B)$ (note that $B$ is symmetric), then for any $\gamma=n_1\gamma_1+n_2\gamma_2$ with $n_1,n_2\in\NN^*$, we have
$$2(\rho|\gamma)_B-(\gamma|\gamma)_B=(n_1-n_1^2)(\alpha|\alpha)+(n_2-n_2^2)(\beta|\beta)-2n_1n_2(\alpha|\beta)\geq -2n_1n_2(\alpha|\beta)>0$$
by assumption and (\ref{eqn:norm_roots}). It then follows from \cite[Propositions~1.7b and 9.11]{Kac} that $\widetilde{\g}'(B)$ is simple modulo center contained in $\hh'_B$. In particular, the restriction of $\phi$ to the free Lie algebra $\widetilde{\nn}^+_B$ is an isomorphism onto its image, yielding the proposition.
\end{proof}

\begin{lemma}\label{lemma:a-bre+}
Let $\alpha,\beta\in\Delta^{im+}$ be such that $(\alpha|\beta)<0$. 
If $\alpha-\beta\in\Delta^{re}$, then $[x,y]\neq 0$ for all nonzero $x\in\g_{\alpha}$ and $y\in\g_{\beta}$.
\end{lemma}
\begin{proof}
Using the $\WW^*$-action, there is no loss of generality in assuming that $\alpha-\beta=\alpha_i$ for some $i\in I$. Let $x\in\g_{\alpha}$ and $y\in\g_{\beta}$ be nonzero, and assume for a contradiction that $[x,y]=0$. Up to normalising $x,y$ (i.e. multiplying them by a nonzero scalar), we may assume by Lemma~\ref{lemma:dualbases} that $x^*:=\omega(x)\in\g_{-\alpha}$ and $y^*:=\omega(y)\in\g_{-\beta}$ satisfy 
$$[x^*,x]=\alpha^{\sharp}\quad\textrm{and}\quad [y^*,y]=\beta^{\sharp}.$$
We also write $$[y^*,x]=\mu e_i\quad\textrm{and}\quad [x^*,y]=\omega([x,y^*])=\mu f_i$$ for some $\mu\in\CC$, so that $\mu\neq 0$ by Lemma~\ref{lemma:basic1}.

Note first that
\begin{equation*}
0=(\ad x^*)[x,y]=(\alpha|\beta)y+\mu [x,f_i].
\end{equation*}
In particular, $[y,[x,f_i]]=0$ and hence
\begin{equation}\label{eqn:fixy=fixy}
0=[f_i,[y,x]]=[[f_i,y],x].
\end{equation}
Note next that
\begin{equation*}
\begin{aligned}
0&=(\ad y^*)^2[y,x]=(\ad y^*)((\alpha|\beta)x+\mu [y,e_i])\\
&=\mu ((\alpha|\beta)+(\beta|\alpha_i))e_i+\mu [y,[y^*,e_i]]\\
&=\mu ((\alpha|\alpha)-(\alpha_i|\alpha_i))e_i+\mu [y,[y^*,e_i]].
\end{aligned}
\end{equation*}
Since $(\alpha|\alpha)-(\alpha_i|\alpha_i)<0$ by (\ref{eqn:norm_roots}), this implies in particular that
\begin{equation}\label{eqn:fixnonzero}
[f_i,y]=\omega([y^*,e_i])\neq 0.
\end{equation}
In particular, $\beta-\alpha_i\in\Delta$. On the other hand, since $(\alpha|\beta)<0$ by assumption, $\alpha$ and $\beta$ are not proportional isotropic roots. Hence $\beta-\alpha_i=2\beta-\alpha$ and $\alpha$ are not proportional isotropic roots either. If $\beta-\alpha_i\in\Delta^{im+}$, we would then have $(\beta-\alpha_i|\alpha)<0$ by Lemma~\ref{lemma:basic_propI}(2) and (\ref{eqn:multiple_roots}), and since 
\begin{equation}\label{eqn:degy-degfix}
\deg(x)-\deg([f_i,y])=\alpha-(\beta-\alpha_i)=2\alpha_i\notin\Delta
\end{equation}
by (\ref{eqn:multiple_roots}), so that $[x,\omega([f_i,y])]=0$, Lemma~\ref{lemma:basic1} would imply that 
$[f_i,y]$ and $x$ generate a free Lie algebra (recall that $[f_i,y]\neq 0$ by (\ref{eqn:fixnonzero})), contradicting (\ref{eqn:fixy=fixy}). Hence  
\begin{equation*}
\gamma:=\beta-\alpha_i\in\Delta^{re+}\quad\textrm{and}\quad \CC [f_i,y]=\CC e_{\gamma}.
\end{equation*}
But then (\ref{eqn:fixy=fixy}) and (\ref{eqn:degy-degfix}) yield
\begin{equation*}
[e_{\gamma},x]=0=[e_{-\gamma},x],
\end{equation*}
so that $(\gamma|\alpha)=0$ by Lemma~\ref{lemma:prop36}(1). Therefore,
$$0=(\gamma|\alpha)=(\beta-\alpha_i|\beta+\alpha_i)=(\beta|\beta)-(\alpha_i|\alpha_i),$$
contradicting (\ref{eqn:norm_roots}).
\end{proof}

\begin{lemma}\label{lemma:a-bim+}
Let $\alpha,\beta\in\Delta^{im+}$ be such that $(\alpha|\beta)<0$. 
If $\alpha-\beta\in\Delta^{im}$, then $[x,y]\neq 0$ for all nonzero $x\in\g_{\alpha}$ and $y\in\g_{\beta}$.
\end{lemma}
\begin{proof}
By Lemma~\ref{lemma:Kac_free_Lie}(1), there is no loss of generality in assuming that $\alpha,\beta$ are non-proportional. We fix a total order $\prec$ on $\Delta^+$ such that $\gamma\prec\gamma'$ whenever $\height(\gamma)<\height(\gamma')$.
Assume for a contradiction that there exist $\alpha,\beta$ as in the statement of the lemma such that $[x,y]= 0$ for some nonzero $x\in\g_{\alpha}$ and $y\in\g_{\beta}$, and take $\alpha+\beta$ minimal for this property. Without loss of generality, we may assume that $\beta\prec \alpha$. Up to normalising $x,y$ (i.e. multiplying them by a nonzero scalar), we may assume by Lemma~\ref{lemma:dualbases} that $x^*:=\omega(x)\in\g_{-\alpha}$ and $y^*:=\omega(y)\in\g_{-\beta}$ satisfy 
$$[x^*,x]=\alpha^{\sharp}\quad\textrm{and}\quad [y^*,y]=\beta^{\sharp}.$$
Finally, we let $n\geq 1$ be maximal such that $\alpha-n\beta\in\Delta^+$.

Note first that
\begin{equation}\label{eqn:1}
0=(\ad y)(\ad x^*)[x,y]=(\ad y)((\alpha|\beta)y+[[y,x^*],x])=[(\ad y)^2x^*,x].
\end{equation}
More generally, 
\begin{equation}\label{eqn:2}
0=[(\ad y)^{m+2}x^*,(\ad y^*)^{m}x]\quad\textrm{for all $m\in\NN$}.
\end{equation}
Indeed, for $m=0$ this is (\ref{eqn:1}), and if (\ref{eqn:2}) holds up to $m\geq 0$ then
\begin{align*}
0&=(\ad y^*)(\ad y)[(\ad y)^{m+2}x^*,(\ad y^*)^{m}x]\\
&=(\ad y^*)\big([(\ad y)^{m+3}x^*,(\ad y^*)^{m}x]+\lambda[(\ad y)^{m+2}x^*,(\ad y^*)^{m-1}x]\big)\\
&=[(\ad y)^{m+3}x^*,(\ad y^*)^{m+1}x]
\end{align*}
for some $\lambda\in\CC$, with the convention that $(\ad y^*)^{m-1}x:=0$ if $m=0$ (here, we used the fact that $[y,x]=[y^*,x^*]=0$ and the induction hypothesis for $m$ and $m-1$). 

In particular, applying (\ref{eqn:2}) to $m=n-1$ and $m=n$, we get
\begin{equation}\label{eqn:3}
[(\ad y)^{n+1}x^*,(\ad y^*)^{n-1}x]=0
\end{equation}
and
\begin{equation}\label{eqn:4}
[(\ad y)^{n+2}x^*,(\ad y^*)^{n}x]=0.
\end{equation}
Note that, by choice of $n$, the elements $(\ad y)^{n+1}x^*$, $(\ad y^*)^{n-1}x$, $(\ad y)^{n+2}x^*$ and $(\ad y^*)^{n}x$ all belong to $\nn^+$ (recall that $\alpha,\beta$ are non-proportional).  
We claim that one of the following four cases must occur:
\begin{enumerate}
\item[(1)]
$(\ad y)^{n+1}x^*=0$.
\item[(2)]
$(n+1)\beta-\alpha\in\Delta^{re+}$.
\item[(3)]
$\alpha-(n-1)\beta\in\Delta^{re+}$.
\item[(4)]
$((n+1)\beta-\alpha|\alpha-(n-1)\beta)=0$.
\end{enumerate}
Indeed, if (1)--(4) do not occur, then $y_{\beta'}:=(\ad y)^{n+1}x^*$ and $x_{\alpha'}:=(\ad y^*)^{n-1}x=\omega((\ad y)^{n-1}x^*)$ are nonzero and have degree $\beta':=(n+1)\beta-\alpha\in\Delta^{im+}$ and $\alpha':=\alpha-(n-1)\beta\in\Delta^{im+}$, respectively. Moreover, $(\alpha'|\beta')<0$ by Lemma~\ref{lemma:basic_propI}(1). In particular, $\alpha'-\beta'\in\Delta$, for otherwise $[\omega(y_{\beta'}),x_{\alpha'}]=0$, and hence Lemma~\ref{lemma:basic1} would imply that $[y_{\beta'},x_{\alpha'}]\neq 0$, contradicting (\ref{eqn:3}). Similarly,  Lemma~\ref{lemma:a-bre+} and (\ref{eqn:3}) imply that $\alpha'-\beta'\in\Delta^{im}$. Since $\alpha'+\beta'= 2\beta\prec\alpha+\beta$, the minimality assumption on $\alpha+\beta$ then yields a contradiction with (\ref{eqn:3}).

We now show that none of the above four cases can occur.

\smallskip
\noindent
{\bf Claim 1:} \emph{Assume that $(2\alpha-s\beta|\beta)=0$ for some $s\in\NN$, and that $\alpha-(s-1)\beta\in\Delta^+$. Then $\alpha-s\beta\notin\Delta^{im+}\cup\Delta^{re}$.}
\\
Indeed, assume for a contradiction that $\alpha-s\beta\in\Delta^{im+}\cup\Delta^{re}$. By assumption,
$$(\alpha-s\beta|\alpha-s\beta)=(\alpha|\alpha)-s(2\alpha-s\beta|\beta)=(\alpha|\alpha),$$
so that $\alpha-s\beta\in\Delta^{im+}$ by (\ref{eqn:norm_roots}). But Lemma~\ref{lemma:basic_propI}(2)  then yields
$$-(\alpha|\beta)=(\alpha-s\beta|\beta)-(2\alpha-s\beta|\beta)=(\alpha-s\beta|\beta)<0,$$
a contradiction.

\smallskip
\noindent
{\bf Claim 2:} \emph{$(\ad y)^{n+1}x^*\neq 0$, and if $(n+1)\beta-\alpha\in\Delta^{re+}$ then $(\ad y)^{n+2}x^*\neq 0$.}
\\
Indeed, let $m\in\{1,\dots,n+1\}$ be maximal such that $(\ad y)^{m}x^*\neq 0$ (if $[y,x^*]=0$ then $[y,x]\neq 0$ by Lemma~\ref{lemma:basic1}, a contradiction). We may assume that $(\ad y)^{m+1}x^*=0$, for otherwise $m=n+1$ and the claim is clear. Then
\begin{align*}
0&=(\ad y^*)(\ad y)^{m+1}x^*=\sum_{i=1}^{m+1}(\beta | (m+1-i)\beta-\alpha)(\ad y)^{m}x^*\\
&=-\frac{m+1}{2}(2\alpha-m\beta|\beta)(\ad y)^{m}x^*,
\end{align*}
so that 
\begin{equation*}
(2\alpha-m\beta|\beta)=0.
\end{equation*}
But as $\alpha-(m-1)\beta\in\Delta^+$ (because $(\ad y)^{m-1}x^*\neq 0$), Claim~1 implies that $\alpha-m\beta\notin\Delta^{im+}\cup\Delta^{re}$. As $\alpha-m\beta\in\Delta^+$ when $m\leq n$ (because $(\ad y)^{m}x^*\neq 0$), 
we deduce that $m=n+1$ and that $(n+1)\beta-\alpha\notin\Delta^{re+}$, yielding the claim.

\smallskip
\noindent
{\bf Claim 3:} \emph{If $(n+1)\beta-\alpha\in\Delta^{re+}$ or if $\alpha-(n-1)\beta\in\Delta^{re+}$, then $\alpha-n\beta\in\Delta^{im+}$.}
\\
Indeed, since $(\ad y)^{n+1}x^*\neq 0$ by Claim~2 (and hence also $(\ad y^*)^{n-1}x\neq 0$), and since $[(\ad y)^{n+1}x^*,(\ad y^*)^{n-1}x]=0$ by (\ref{eqn:3}), Lemma~\ref{lemma:prop36}(1) implies that
$$0\leq ((n+1)\beta-\alpha|\alpha-(n-1)\beta)=(\beta|\beta)-(\alpha-n\beta| \alpha-n\beta),$$
so that the claim follows from (\ref{eqn:norm_roots}).

\smallskip
\noindent
{\bf Claim 4:} \emph{If $(n+1)\beta-\alpha\in\Delta^{re+}$, then $(n+2)\beta-\alpha\notin\Delta^{re+}$.}
\\
Indeed, assume for a contradiction that $(n+2)\beta-\alpha\in\Delta^{re+}$. Since $(\ad y)^{n+2}x^*\neq 0$ by Claim~2 (and hence also $(\ad y^*)^{n}x\neq 0$), and since $[(\ad y)^{n+2}x^*,(\ad y^*)^{n}x]=0$ by (\ref{eqn:4}), Lemma~\ref{lemma:prop36}(1) then implies that
$$0\leq ((n+2)\beta-\alpha|\alpha-n\beta)=(\beta|\beta)-((n+1)\beta-\alpha| (n+1)\beta-\alpha),$$
contradicting (\ref{eqn:norm_roots}).

\medskip
We can now prove that the cases (1)--(4) cannot occur. For case (1), this follows from Claim~2.

In case (2), Claims~2 and 4 imply that $(\ad y)^{n+2}x^*\neq 0$ and $(n+2)\beta-\alpha\in\Delta^{im+}$, and Claim~3 yields $\alpha-n\beta\in\Delta^{im+}$. Since, moreover,
$$((n+2)\beta-\alpha|\alpha-n\beta)=(\beta|\beta)-((n+1)\beta-\alpha|(n+1)\beta-\alpha)<0$$
by (\ref{eqn:norm_roots}), whereas the difference between $(n+2)\beta-\alpha$ and $\alpha-n\beta$ is not a root by (\ref{eqn:multiple_roots}), Lemma~\ref{lemma:basic1} implies that $[(\ad y)^{n+2}x^*,(\ad y^*)^{n}x]\neq 0$, contradicting (\ref{eqn:4}).

In case (3), there exists some $w\in\WW$ such that $w\alpha-(n-1)w\beta=\alpha_i$ for some $i\in I$. On the other hand, Claim~3 implies that $\alpha-n\beta\in\Delta^{im+}$ and hence $$\alpha_i-w\beta=w(\alpha-n\beta)\in\Delta^{im+},$$ contradicting the fact that $w\beta\in\Delta^{im+}$.

Finally, since the cases (1), (2) and (3) cannot occur, we deduce that $(n+1)\beta-\alpha,\alpha-(n-1)\beta\in\Delta^{im+}$. Since $2\beta\in\Delta^{im+}$ by (\ref{eqn:multiple_roots}), Lemma~\ref{lemma:basic_propI}(2) then implies that
$$((n+1)\beta-\alpha|\alpha-(n-1)\beta)<0,$$
and hence case (4) cannot occur either, as desired.
\end{proof}

\begin{theorem}\label{thm:pre-free}
Let $\alpha,\beta\in\Delta$. If $(\alpha|\beta)<0$ then $[x,y]\neq 0$ for every nonzero $x\in\g_{\alpha}$ and $y\in\g_{\beta}$ such that $\CC x\neq\CC y$.
\end{theorem}
\begin{proof}
If $\alpha$ or $\beta$ is a real root, this follows from Lemma~\ref{lemma:prop36}(1). Assume now that $\alpha,\beta\in\Delta^{im}$. Then Lemma~\ref{lemma:basic_propI}(1) implies that either $\alpha,\beta\in\Delta^{im+}$ or $\alpha,\beta\in\Delta^{im-}$, and there is no loss of generality in assuming that $\alpha,\beta\in\Delta^{im+}$ (using the action of $\omega$). By Lemma~\ref{lemma:Kac_free_Lie}(1), we may moreover assume that $\alpha\neq\beta$. The theorem then follows in that case from Lemma~\ref{lemma:basic1}, Lemma~\ref{lemma:a-bre+}, or Lemma~\ref{lemma:a-bim+}, depending on whether $\alpha-\beta\notin\Delta$, $\alpha-\beta\in\Delta^{re}$, or $\alpha-\beta\in\Delta^{im}$.
\end{proof}

\begin{lemma}\label{lemma:strict_inequality}
Let $\alpha,\beta\in\Delta^{im+}$ with $(\alpha|\beta)<0$. Let $Y$ be a nonzero subspace of $\g_{\beta}$, and let $x,x'\in\g_{\alpha}$ be nonzero. If $[x',Y]\subseteq [x,Y]$, then $\CC x'=\CC x$.
\end{lemma}
\begin{proof}
Write $Y=\bigoplus_{i=1}^n\CC y_i$ for some linearly independent $y_i\in\g_{\beta}$. Consider the complex matrix $\Lambda=(\lambda_{ij})_{1\leq i,j\leq n}$ defined by $$[x',y_i]=\sum_{j=1}^n\lambda_{ij}[x,y_j]\quad\textrm{for all $i=1,\dots,n$.}$$
In view of Theorem~\ref{thm:pre-free}, it is sufficient to show that the equation
$$[x'-ax,\sum_{i=1}^n\mu_iy_i]=0$$
admits a nontrivial solution $(a,\mu_1,\dots,\mu_n)\in\CC^{n+1}$, in the sense that $(\mu_1,\dots,\mu_n)\neq (0,\dots,0)$. This equation can be rewritten as
$$\sum_{i=1}^n\sum_{j=1}^n\mu_i\lambda_{ij}[x,y_j]=\sum_{i=1}^na\mu_i[x,y_i],$$
and it is thus sufficient to find a nontrivial solution to the equation
$$\sum_{i=1}^n\sum_{j=1}^n\mu_i\lambda_{ij}y_j=\sum_{i=1}^na\mu_iy_i.$$
In turn, this last equation is equivalent to the system of equations
$$\sum_{i=1}^n\mu_i\lambda_{ij}=a\mu_j\quad\textrm{for $j=1,\dots,n$,}$$
which we can rewrite as
$$\Lambda^T\mu=a\mu, \quad\textrm{where $\mu:=\begin{psmallmatrix}\mu_1\\ \vdots \\ \mu_n\end{psmallmatrix}$.}$$
In other words, it is sufficient to find a (nonzero) eigenvector $\mu$ of $\Lambda^T$, which of course always exists over $\CC$.
\end{proof}

\begin{corollary}
Let $\alpha,\beta\in\Delta^{im+}$ with $\alpha\neq\beta$. Then one of the following holds:
\begin{enumerate}
\item
$[\g_{\alpha},\g_{\beta}]=\{0\}$. In this case, either $\alpha+\beta\notin\Delta$ or $\alpha,\beta$ are proportional isotropic roots.
 \item
 $\dim [\g_{\alpha},\g_{\beta}]\geq\max\{\dim\g_{\alpha},\dim\g_{\beta}\}$, with equality if and only if \\ $\min\{\dim\g_{\alpha},\dim\g_{\beta}\}=1$.
 \end{enumerate}
\end{corollary}
\begin{proof}
If $\alpha+\beta\notin\Delta$, or if $\alpha,\beta$ are proportional isotropic roots, then $[\g_{\alpha},\g_{\beta}]=\{0\}$ (see Lemma~\ref{lemma:Kac_free_Lie}(2)). We may thus assume by Lemma~\ref{lemma:basic_propI}(2) that $(\alpha|\beta)<0$. Up to permuting $\alpha$ and $\beta$, we may moreover assume that $\dim\g_{\alpha}\leq \dim\g_{\beta}$. Let $x\in\g_{\alpha}$ be nonzero. Theorem~\ref{thm:pre-free} then implies that $\dim [\g_{\alpha},\g_{\beta}]\geq\dim [x,\g_{\beta}]=\dim\g_{\beta}$, with equality if $\dim\g_{\alpha}=1$. On the other hand, if $\dim\g_{\alpha}\geq 2$, then choosing some $x'\in\g_{\alpha}\setminus\CC x$, we deduce from Lemma~\ref{lemma:strict_inequality} that $[x',\g_{\beta}]\not\subseteq [x,\g_{\beta}]$, so that $\dim [\g_{\alpha},\g_{\beta}]>\dim [x,\g_{\beta}]=\dim\g_{\beta}$, as desired.
\end{proof}

\section{Solvable and nilpotent subalgebras}
We fix again a symmetrisable GCM $A$, and keep all notations from \S\ref{section:preliminaries}.
In this section, we characterise the solvable and nilpotent graded subalgebras of $\g(A)$. We recall that a Lie algebra $\LLL$ is {\bf solvable} (resp. {\bf nilpotent}) if $\LLL^{(n)}=\{0\}$ (resp. $\LLL^n=\{0\}$) for some $n\in\NN$, where the subalgebras $\LLL^{(n)}$ and $\LLL^n$ of $\LLL$ are defined recursively by
$$\LLL^{(0)}=\LLL^0:=\LLL, \quad \LLL^{(n+1)}:=[\LLL^{(n)},\LLL^{(n)}], \quad\textrm{and}\quad \LLL^{n+1}:=[\LLL,\LLL^n]\quad\textrm{for all $n\in\NN$.}$$

 We also recall that, given a subalgebra $\LLL$ of $\g(A)$, an element $x\in\LLL$ is called {\bf $\ad$-locally nilpotent on $\LLL$} if for every $y\in\LLL$, there exists some $N=N(y)\in\NN$ such that $(\ad x)^Ny=0$. More generally, $x\in\LLL$ is called {\bf $\ad$-locally finite on $\LLL$} if for every $y\in\LLL$, there exists a finite-dimensional subspace $V=V(y)\subseteq\LLL$ containing $y$ and such that $[x,V]\subseteq V$. For instance, the Chevalley generators $e_i,f_i$ ($i\in I$) are $\ad$-locally nilpotent on $\g(A)$, while the elements of the Cartan subalgebra $\hh$ are $\ad$-locally finite on $\g(A)$.

\begin{lemma}\label{lemma:symm1}
Let $\beta\in\Delta^{im+}$ and $\gamma\in\Delta^{re}$ be such that $\beta\pm\gamma\in\Delta^{re}$. Then $(\beta|\gamma)=0$. If, moreover, $\beta\in\Delta^{im+}_{an}$, then for any $x\in\g_{\beta}$, the following assertions are equivalent:
\begin{enumerate}
\item
$[x,e_{\gamma-\beta}]\neq 0$.
\item
$[x,e_{\gamma}]\neq 0$.
\item
$(\ad x)^ne_{\gamma-\beta}\neq 0$ for all $n\in\NN$.
\end{enumerate}
\end{lemma}
\begin{proof}
Note that
$$0<\frac{(\beta\pm\gamma|\beta\pm\gamma)}{(\gamma|\gamma)}=\frac{(\beta|\beta)}{(\gamma|\gamma)}\pm 2\frac{(\beta|\gamma)}{(\gamma|\gamma)}+1=\frac{(\beta|\beta)}{(\gamma|\gamma)}\pm \beta(\gamma^{\vee})+1$$
by (\ref{eqn:coroot}) and (\ref{eqn:norm_roots}). Thus
$$0\geq\frac{(\beta|\beta)}{(\gamma|\gamma)}>|\beta(\gamma^{\vee})|-1$$
by (\ref{eqn:norm_roots}), so that $\beta(\gamma^{\vee})=0=(\beta|\gamma)$ (recall that $\beta(\gamma^{\vee})\in\ZZ$ by Lemma~\ref{lemma:RS}). This proves the first claim.

Assume now that $(\beta|\beta)< 0$ and let $x\in\g_{\beta}$. Set $\alpha:=\gamma-\beta\in\Delta^{re}$.
If $[x,e_{\alpha}]=0$, then for any $z\in\g_{-\beta}$,  (\ref{eqn:prop_inv_form}) yields
$$[z,[x,e_{\gamma}]]\in \CC (\beta|\gamma)e_{\gamma}+\CC [x,e_{\alpha}]=\{0\},$$
that is, $[[x,e_{\gamma}],\g_{-\beta}]=\{0\}$. Thus, if $[x,e_{\alpha}]=0$, then $[x,e_{\gamma}]=0$ by Lemma~\ref{lemma:prop36}(2) as $\deg([x,e_{\gamma}])+(-\beta)=\gamma\in\Delta$. In other words, if we consider the linear maps 
$$u_1\co\g_{\beta}\to\g_{\gamma}:z\mapsto [e_{\alpha},z]\quad\textrm{and}\quad u_2\co\g_{\beta}\to\g_{\gamma+\beta}:z\mapsto [e_{\gamma},z],$$
then $\ker u_1\subseteq\ker u_2\subseteq\g_{\beta}$. On the other hand, note that $u_1$ and $u_2$ are nonzero by Lemma~\ref{lemma:prop36}(2). Since $\dim(\g_{\gamma})=1$, we deduce that $\ker u_1$ has codimension $1$ in $\g_{\beta}$ and hence that $\ker u_1=\ker u_2$. This proves the equivalence of (1) and (2). Finally, if (1) (and hence also (2)) holds, we already know that $(\ad x)^ne_{\alpha}\neq 0$ for $n=1,2$. That it holds for all $n\in\NN$ then follows from Theorem~\ref{thm:pre-free} since $(\beta|\alpha+n\beta)=(n-1)(\beta|\beta)<0$ for all $n\in\NN$ with $n\geq 2$ by assumption.
\end{proof}

\begin{lemma}\label{lemma:isotropic+real}
Let $\alpha\in\Delta^{re}$ and $\beta\in\Delta_{is}^{im+}$ be such that $(\alpha|\beta)\geq 0$. Let $w\in\WW$ be such that $w\beta\in K_0$. Then the following assertions hold:
\begin{enumerate}
\item
Either $\supp(w\alpha)\subseteq\supp(w\beta)$, or $\supp(w\alpha)\cup\supp(w\beta)$ is not connected.
\item
If $x\in\g_{\alpha}$ and $y\in\g_{\beta}$ are such that $[x,y]\neq 0$, then $(\ad y)^nx\neq 0$ for all $n\in\NN$.
\end{enumerate}
\end{lemma}
\begin{proof}
Using the $\WW^*$-action, we may assume that $\beta\in K_0$ (i.e. $w=1$), so that $J:=\supp(\beta)\subseteq I$ is of affine type by (\ref{eqn:isotropic}). In particular, $(\beta|\gamma)=0$ for all $\gamma\in Q_+$ with $\supp(\gamma)\subseteq J$ by (\ref{eqn:isotropic_zero}). Write $\alpha=\alpha_J+\alpha'$ with $\supp(\alpha_J)\subseteq J$ and $J':=\supp(\alpha')\subseteq I\setminus J$. By assumption, $(\beta|\alpha')=(\beta|\alpha)\geq 0$. Writing $\beta=\sum_{j\in J}k_j\alpha_j$ ($k_j>0$) and $\alpha'=\sum_{i\in J'}k_i'\alpha_i$ ($k_i'>0$), we thus have
$$0\leq \sum_{j\in J}\sum_{i\in J'}k_i'k_j(\alpha_i|\alpha_j),$$
so that $(\alpha_i|\alpha_j)=0$ for all $i\in J'$ and $j\in J$ (recall that $(\alpha_i|\alpha_j)\leq 0$ for all $i\neq j$). But since $\supp(\alpha)$ is connected, this implies that either $\alpha_J=0$ or $\alpha'=0$, yielding (1). 

Let now $x\in\g_{\alpha}$ and $y\in\g_{\beta}$ be such that $[x,y]\neq 0$. In particular, $\alpha+\beta\in\Delta$ and hence $\supp(\alpha)\cup\supp(\beta)$ is connected, so that $\supp(\alpha)\subseteq\supp(\beta)$ by (1). Thus $\alpha,\beta$ are roots of the affine Kac--Moody algebra with GCM $(a_{ij})_{i,j\in J}$, and (2) follows from the realisation of affine Kac--Moody algebras as (twisted) loop algebras over a simple finite-dimensional Lie algebra $\mathring{\g}$ with Cartan subalgebra $\mathring{\hh}\subseteq\hh$ (see \cite[Theorems~7.4 and 8.3]{Kac}): the element $y$ of $\g_{\beta}$ is of the form $y=t^m\otimes h$ for some $h\in\mathring{\hh}$ and $m\in\NN^*$ ($t$ being the indeterminate in the loop algebra), and hence $(\ad y)^nx=t^{mn}\otimes \alpha(h)^nx$ for all $n\in\NN$. Thus $(\ad y)^nx\neq 0$ ($n\in\NN^*$) if and only if $\alpha(h)\neq 0$ if and only if $[y,x]\neq 0$, as desired.
\end{proof}

\begin{theorem}\label{thm:adx^ny}
Let $\alpha\in\Delta^+\cup\Delta^{re}$ and $\beta\in\Delta^{im+}$. Let $x\in\g_{\alpha}$ and $y\in\g_{\beta}$ be such that $[x,y]\neq 0$. Then $(\ad y)^nx\neq 0$ for all $n\in\NN$.
\end{theorem}
\begin{proof}
Using the $\WW^*$-action, there is no loss of generality in assuming that $\alpha\in\Delta^+$. Note also that
\begin{equation}\label{eqn:basic_case}
\textrm{$(\ad y)^nx\neq 0$ for all $n\in\NN$ if $(\alpha|\beta)<0$.}
\end{equation}
Indeed, if $(\alpha|\beta)<0$, then $(\beta|\alpha+n\beta)\leq (\beta|\alpha)<0$ for all $n\in\NN$ by (\ref{eqn:norm_roots}), so that the conclusion follows inductively on $n$ from Theorem~\ref{thm:pre-free}.

Assume first that $\alpha\in\Delta^{im+}$. Since $[x,y]\neq 0$, Lemma~\ref{lemma:Kac_free_Lie}(2) implies that $\alpha,\beta$ are not proportional isotropic roots. As $\alpha+\beta\in\Delta^+$ (again because $[x,y]\neq 0$), Lemma~\ref{lemma:basic_propI}(2) then implies that $(\alpha|\beta)<0$, so that the claim follows from (\ref{eqn:basic_case}).

Assume next that $\alpha\in\Delta^{re+}$ and that $\beta\in\Delta^{im+}_{is}$.  If $(\alpha|\beta)<0$, the claim follows from (\ref{eqn:basic_case}). On the other hand, if $(\alpha|\beta)\geq 0$, the claim follows from Lemma~\ref{lemma:isotropic+real}(2).

Finally, assume that $\alpha\in\Delta^{re+}$ and that $\beta\in\Delta^{im+}_{an}$. Set $\gamma:=\alpha+\beta=\deg([x,y])\in\Delta^+$. If $(\beta|\gamma)<0$, then $(\beta|\gamma+n\beta)<0$ for all $n\in\NN^*$ by (\ref{eqn:norm_roots}), so that the claim follows inductively on $n$ from Theorem~\ref{thm:pre-free}. We may thus assume that $(\beta|\gamma)\geq 0$. If $\gamma\in\Delta^{im+}$, then Lemma~\ref{lemma:basic_propI}(1,2) implies that $(\beta|\gamma)=0$ and $\beta+\gamma\notin\Delta$. Hence, in that case, Lemma~\ref{lemma:imaginary_not_connected} yields some $w\in\WW$ such that $\supp(w\beta)\cup\supp(w\gamma)=\supp(w\beta)\cup\supp(w\alpha)$ is not connected, contradicting the fact that $w(\alpha+\beta)\in\Delta$.
Thus $\gamma\in\Delta^{re+}$. Hence $\beta+\gamma\in\Delta^+$ (this is because $\beta-\gamma\in\Delta^{re}$ and $\beta\in\Delta^{im+}$, see Lemma~\ref{lemma:RS}), and since
$$(\beta+\gamma|\beta+\gamma)=(\beta-\gamma|\beta-\gamma)+4(\beta|\gamma)=(\alpha|\alpha)+4(\beta|\gamma)>0$$
by (\ref{eqn:norm_roots}), we conclude that $\beta+\gamma\in\Delta^{re+}$ by (\ref{eqn:norm_roots}). But then $\beta\pm\gamma\in\Delta^{re}$, so that $(\ad y)^nx\neq 0$ for all $n\in\NN$ by Lemma~\ref{lemma:symm1}, as desired. 
\end{proof}

\begin{lemma}\label{lemma:xy_non_solvable}
Let $\alpha\in\Delta^{im}$ and $\beta\in\Delta$ be such that $(\alpha|\beta)<0$. Let $x\in\g_{\alpha}$ and $y\in\g_{\beta}$ be nonzero and such that $\CC x\neq \CC y$. Then the subalgebra of $\g(A)$ generated by $x$ and $y$ is not solvable.
\end{lemma}
\begin{proof}
Up to using the action of $\omega$, we may assume that $\alpha\in\Delta^{im+}$. Thus $\beta\notin\Delta^{im-}$ by Lemma~\ref{lemma:basic_propI}(1). Since $[x,y]\neq 0$ by Theorem~\ref{thm:pre-free}, it then follows from Theorem~\ref{thm:adx^ny} that $(\ad x)^ny\neq 0$ for all $n\in\NN$.

We define a bracket map $\bra\cdot\ket\co \bigcup_{n\in\NN}\g(A)^{2^n}\to\g(A)$ recursively on $n\in\NN$ (where $\g(A)^m:=\g(A)\times\dots\times\g(A)$, $m$ factors), by setting $\bra x \ket :=x$ for all $x\in\g(A)$, and 
$$\bra x_1,\dots,x_{2^n}\ket:=\big[\bra x_1,\dots,x_{2^{n-1}}\ket ,\bra x_{2^{n-1}+1},\dots,x_{2^n}\ket\big]\quad\textrm{for all  $x_1,\dots,x_{2^n}\in\g(A)$.}$$
Since $(\alpha|\alpha)\leq 0$ by (\ref{eqn:norm_roots}) and $(\alpha|\beta)<0$  by hypothesis, there exists some $r\in\NN^*$ such that
$$(r\alpha+\beta| r\alpha+\beta)<0.$$
 For each $n\in\NN$, set $y_n:=(\ad x)^{rn}y\neq 0$. We now show inductively on $n$ that
$$z_n:=\bra y_{i_1},\dots,y_{i_{2^n}}\ket\neq 0\quad\textrm{for all $n\in\NN$ and all $i_1<i_2<\dots<i_{2^n}$},$$
so that the subalgebra generated by $x$ and $y$ is indeed not solvable. For $n=0$, this is clear. Let now $n\in\NN^*$ and $i_1,\dots,i_{2^n}\in\NN^*$ with $i_1<i_2<\dots<i_{2^n}$. By induction hypothesis, 
$$z_n^1:=\bra y_{i_1},\dots,y_{i_{2^{n-1}}}\ket\neq 0\quad\textrm{and}\quad z_n^2:=\bra y_{i_{2^{n-1}+1}},\dots,y_{i_{2^n}}\ket\neq 0.$$
Note that
$$i_{(1)}:=i_1+\dots+i_{2^{n-1}}\geq 2^{n-1}\quad\textrm{and}\quad i_{(2)}:=i_{2^{n-1}+1}+\dots+i_{2^n}\geq 2^{n-1}.$$ Since $(\alpha|\alpha)\leq 0$ and $(\alpha|\beta)<0$, we deduce that
$$\big(\deg(z_n^1) \big| \deg(z_n^2)\big)=\big(ri_{(1)}\alpha+2^{n-1}\beta \big| ri_{(2)}\alpha+2^{n-1}\beta\big)\leq 2^{2n-2}(r\alpha+\beta| r\alpha+\beta)<0.$$
Since, moreover, $\deg(z_n^1)\neq\deg(z_n^2)$ (because $i_{(1)}<i_{(2)}$), Theorem~\ref{thm:pre-free} implies that $z_n=[z_n^1,z_n^2]\neq 0$, thus completing the induction step.
\end{proof}

\begin{lemma}\label{lemma:solvable_hh}
Let $\LLL$ be a solvable graded subalgebra of $\g(A)$. Then $[\hh\cap \LLL^1,\LLL]=\{0\}$.
\end{lemma}
\begin{proof}
Assume for a contradiction that there exists some $h\in\hh\cap [\LLL,\LLL]$ and some $x\in\LLL\cap\g_{\alpha}$ ($\alpha\in\Delta$) such that $[h,x]\neq\{0\}$. Then $h$ is of the form $$h=\sum_{i=1}^m[x_{-\beta_i},x_{\beta_i}]$$ for some $\beta_1,\dots,\beta_m\in\Delta^+$ and some $x_{\pm\beta_i}\in\LLL\cap\g_{\pm\beta_i}$ with $[x_{-\beta_i},x_{\beta_i}]\neq 0$. Note that $\beta_1,\dots,\beta_m\in\Delta_{is}^{im+}$, for otherwise $\LLL$ would contain a copy $\CC x_{-\beta_i}\oplus\CC \beta_i^{\sharp}\oplus\CC x_{\beta_i}$ of $\sll_2(\CC)$ (see (\ref{eqn:prop_inv_form})). Since 
$$0\neq [h,x]\in\sum_{i=1}^m\CC[\beta_i^{\sharp},x]=\sum_{i=1}^m\CC(\beta_i|\alpha)x$$
by (\ref{eqn:prop_inv_form}), there exists some $r\in\{1,\dots,m\}$ such that $(\beta_r|\alpha)\neq 0$. In particular, $\alpha\neq\pm\beta_r$. Since $x_{\pm\beta_r}$ and $x$ generate a solvable subalgebra, Lemma~\ref{lemma:xy_non_solvable} then implies that $(\alpha|\beta_r)\geq 0$ and $(\alpha|-\beta_r)\geq 0$, yielding the desired contradiction.
\end{proof}

\begin{lemma}\label{lemma:solvable_affine}
Let $\alpha,\gamma\in\Delta^{re}$ be such that $\beta:=\alpha+\gamma\in\Delta^{im+}$ and $(\alpha|\beta)=0=(\gamma|\beta)$. Then the subalgebra generated by $e_{\alpha},e_{\gamma}$ is not solvable and $[e_{\alpha},[e_{\alpha},e_{\gamma}]]\neq 0$.
\end{lemma}
\begin{proof}
Let $\LLL$ denote the subalgebra of $\g(A)$ generated by $e_{\alpha},e_{\gamma}$. By assumption, we have
$$(\alpha|\alpha)=-(\alpha|\gamma)=(\gamma|\gamma).$$
In particular,
\begin{equation}\label{eqn:rel-2}
\la \gamma,\alpha^{\vee}\ra =\frac{2(\gamma|\alpha)}{(\alpha|\alpha)}=-2=\frac{2(\alpha|\gamma)}{(\gamma|\gamma)}=\la\alpha,\gamma^{\vee}\ra
\end{equation}
by (\ref{eqn:coroot}). 

Assume first that $\gamma-\alpha\notin\Delta$, so that
\begin{equation}\label{eqn:rel11}
[e_{\alpha},e_{-\gamma}]=0=[e_{-\alpha},e_{\gamma}].
\end{equation}
Then $S(\alpha,\gamma)=\{\gamma,\gamma+\alpha,\gamma+2\alpha\}$ and $S(\gamma,\alpha)=\{\alpha,\alpha+\gamma,\alpha+2\gamma\}$ by Lemma~\ref{lemma:RS}, so that
\begin{equation}\label{eqn:rel12}
(\ad e_{\alpha})^3e_{\gamma}=(\ad e_{-\alpha})^3e_{-\gamma}=(\ad e_{\gamma})^3e_{\alpha}=(\ad e_{-\gamma})^3e_{-\alpha}=0.
\end{equation}
Consider the GCM $B=\begin{psmallmatrix}2&-2\\ -2 &2\end{psmallmatrix}$. Denoting by $e_1^B,e_2^B$ and $f_1^B,f_2^B$ the Chevalley generators of $\g'(B)\subseteq\g(B)$, the assignment
$$e_1^B\mapsto e_{\alpha}, \quad e_2^B\mapsto e_{\gamma}, \quad f_1^B\mapsto e_{-\alpha}, \quad\textrm{and}\quad  f_2^B\mapsto e_{-\gamma}$$
defines a Lie algebra morphism $\phi\co\g'(B)\to\g(A)$ (see \S\ref{subsection:KMA}), as follows from the relations (\ref{eqn:sl2_alpha}), (\ref{eqn:rel-2}), (\ref{eqn:rel11}) and (\ref{eqn:rel12}). Since $\g'(B)$ is simple modulo center (see \S\ref{subsection:KMA}), the restriction of $\phi$ to $\nn^+(B)$ then defines an isomorphism $\nn^+(B)\cong \LLL$. In particular, $[e_{\alpha},[e_{\alpha},e_{\gamma}]]\neq 0$. Moreover, it follows from (\ref{eqn:presentation_nn+}) that the assignment
$$e_1^B\mapsto\begin{psmallmatrix}0&1\\ 0 &0\end{psmallmatrix}\quad\textrm{and}\quad e_2^B\mapsto\begin{psmallmatrix}0&0\\ -1 &0\end{psmallmatrix}$$
defines a surjective Lie algebra morphism $\nn^+(B)\to\sll_2(\CC)$, and hence $\LLL$ is not solvable.

Assume next that $\gamma-\alpha\in\Delta$. Then $\gamma-\alpha\in\Delta^{re}$ and $\gamma-2\alpha\notin\Delta$ by Lemma~\ref{lemma:RS}, so that 
\begin{equation}\label{eqn:rel21}
[e_{\alpha},e_{\alpha-\gamma}]=0=[e_{-\alpha},e_{\gamma-\alpha}].
\end{equation}
Note also that
\begin{equation}\label{eqn:rel-23}
\la \gamma-\alpha,\alpha^{\vee}\ra=-4\quad\textrm{and}\quad \la \alpha,(\gamma-\alpha)^{\vee}\ra=\frac{2(\alpha|\gamma-\alpha)}{(\gamma-\alpha|\gamma-\alpha)}=\frac{4(\alpha|\gamma)}{-4(\alpha|\gamma)}=-1
\end{equation}
by (\ref{eqn:rel-2}). Lemma~\ref{lemma:RS} then implies that $S(\alpha,\gamma-\alpha)=\{\gamma+n\alpha \ | \ -1\leq n\leq 3\}$ and that $S(\gamma-\alpha,\alpha)=\{\alpha,\gamma\}$, and hence
\begin{equation}\label{eqn:rel22}
(\ad e_{\alpha})^5e_{\gamma-\alpha}=(\ad e_{-\alpha})^5e_{\alpha-\gamma}=(\ad e_{\gamma-\alpha})^2e_{\alpha}=(\ad e_{\alpha-\gamma})^2e_{-\alpha}=0.
\end{equation}
Consider the GCM $C=\begin{psmallmatrix}2&-4\\ -1 &2\end{psmallmatrix}$. Denoting by $e_1^C,e_2^C$ and $f_1^C,f_2^C$ the Chevalley generators of $\g'(C)\subseteq\g(C)$, the assignment
$$e_1^C\mapsto e_{\alpha}, \quad e_2^C\mapsto e_{\gamma-\alpha}, \quad f_1^C\mapsto e_{-\alpha}, \quad\textrm{and}\quad  f_2^C\mapsto e_{\alpha-\gamma}$$
defines a Lie algebra morphism $\psi\co\g'(C)\to\g(A)$ (see \S\ref{subsection:KMA}), as follows from the relations (\ref{eqn:sl2_alpha}), (\ref{eqn:rel21}), (\ref{eqn:rel-23}) and (\ref{eqn:rel22}). As before, the restriction of $\psi$ to $\nn^+(C)$ is injective. Since $\CC e_{\gamma}=\CC [e_{\alpha},e_{\gamma-\alpha}]$ by Lemma~\ref{lemma:rerere}, this implies that $\psi$ restricts to an isomorphism $\LLL'\cong\LLL$ from the Lie subalgebra $\LLL'$ of $\nn^+(C)$ generated by $e_1^C,[e_1^C,e_2^C]$ to $\LLL$. In particular, $[e_{\alpha},[e_{\alpha},e_{\gamma}]]\neq 0$. Moreover, $\LLL'$ (and hence $\LLL$) is not solvable: in the notations of \cite[Exercises~8.15 and 8.16]{Kac}, $\LLL'$ can be identified with the subalgebra of the twisted loop algebra $\LLL(\sll_3(\CC),\mu)$ generated by $\CC L_1=\CC e_1^C$ and $\CC L_3=\CC [e_1^C,e_2^C]$ (see also \cite[Example~5.27]{KMGbook}). Straightforward computations using \cite[Exercises~8.16]{Kac} show that $\LLL'$ contains $\CC L_7=\CC [L_3,[L_3,L_1]]$ and $\CC L_8=\CC [L_1,L_7]$, and that 
$$[L_1,L_7]=L_8, \quad [L_8,L_1]=t^2L_1, \quad [L_8,L_7]=-t^2L_7,$$
so that the subalgebra generated by $L_1,L_7,L_8$ (and hence also $\LLL'$) is not solvable.
\end{proof}

\begin{lemma}\label{lemma:real_part_closed}
Let $\LLL$ be a graded subalgebra of $\g(A)$ such that each homogeneous element of $\LLL$ is $\ad$-locally finite on $\LLL$. Then $\Psi:=\{\alpha\in\Delta^{re} \ | \ \LLL\cap\g_{\alpha}\neq\{0\}\}$ is a closed set of roots.
\end{lemma}
\begin{proof}
In view of  Lemma~\ref{lemma:rerere}, it is sufficient to show that if $\alpha,\gamma\in\Psi$ then $\alpha+\gamma\notin\Delta^{im}$. Assume for a contradiction that $\beta:=\alpha+\gamma\in\Delta^{im+}$ for some $\alpha,\gamma\in\Psi$ (the case $\beta\in\Delta^{im}_-$ will then also follow, using the action of $\omega$). Up to conjugating $\LLL$ by some element of $\WW^*$, there is no loss of generality in assuming that $\gamma$ is a simple root, and hence that $\alpha,\gamma\in\Delta^{re+}$.

As
$$2(\alpha|\gamma)=(\alpha+\gamma|\alpha+\gamma)-(\alpha|\alpha)-(\gamma|\gamma)<0$$
by (\ref{eqn:norm_roots}), Lemma~\ref{lemma:prop36}(1) yields $[e_{\alpha},e_{\gamma}]\neq 0$. 
On the other hand, $\gamma+2\alpha\in\Delta$ and $\alpha+2\gamma\in\Delta$ by Lemma~\ref{lemma:RS}. If $\gamma+2\alpha\in\Delta^{im+}$, then  
$$4(\gamma+\alpha|\alpha)=(\gamma+2\alpha|\gamma+2\alpha)-(\gamma|\gamma)<0$$
by (\ref{eqn:norm_roots}), and hence $[e_{\alpha},[e_{\alpha},e_{\gamma}]]\neq 0$ by Lemma~\ref{lemma:prop36}(1), so that $[e_{\alpha},e_{\gamma}]\in\LLL$ is not $\ad$-locally finite on $\LLL$ by Theorem~\ref{thm:adx^ny}, a contradiction. Thus $\gamma+2\alpha\in\Delta^{re+}$ and, similarly, $\alpha+2\gamma\in\Delta^{re+}$. Lemma~\ref{lemma:symm1} then yields that
$$(\beta|\alpha)=0=(\beta|\gamma),$$
and hence $[e_{\alpha},[e_{\alpha},e_{\gamma}]]\neq 0$ by Lemma~\ref{lemma:solvable_affine}, again contradicting Theorem~\ref{thm:adx^ny}.
\end{proof}

We are now ready to describe the graded subalgebras of $\g(A)$ all whose elements are $\ad$-locally finite. Let us remark that a subalgebra of $\g(A)$ that contains $\hh$ is automatically graded by \cite[Proposition~1.5]{Kac}.

\begin{theorem}\label{thm:subgA}
Let $\LLL$ be a graded subalgebra of $\g(A)$ such that each homogeneous element of $\LLL$ is $\ad$-locally finite on $\LLL$. Then there exists a closed set of real roots $\Psi\subseteq\Delta^{re}$, and abelian subalgebras $\LLL_0\subseteq\hh$, $\LLL^{im+}\subseteq\nn^{im+}$ and $\LLL^{im-}\subseteq\nn^{im-}$ such that 
\begin{enumerate}
\item
$\LLL=\LLL_0\oplus\g_{\Psi}\oplus\LLL^{im+}\oplus\LLL^{im-}$;
\item
$[\g_{\Psi},\LLL^{im+}]=\{0\}=[\g_{\Psi},\LLL^{im-}]$;
\item
$[\LLL^{im+},\LLL^{im-}]\subseteq \LLL_0\oplus\g_{\Psi}$.
\end{enumerate}
Moreover, denoting by $\LLL^{im}$ the subalgebra generated by $\LLL^{im+}\oplus\LLL^{im-}$, we have the following equivalences:
\begin{enumerate}
\item[(4)]
$\g_{\Psi}$ is a nilpotent subalgebra $\iff$ $\Psi$ does not contain any pair of opposite roots.
\item[(5)]
$\LLL^{im}$ is nilpotent $\iff$ $[\LLL^{im},\LLL_0\cap\LLL^{im}]=\{0\}$ $\iff$ $\LLL^{im}$ is solvable. In that case, $\LLL^{im}$ is nilpotent of degree at most $2$.
\item[(6)]
$\LLL$ is solvable $\iff$ $\LLL^{im}$ and $\g_{\Psi}$ are nilpotent subalgebras $\iff$ $[\LLL,\LLL]$ is nilpotent.
\item[(7)]
$\LLL$ is nilpotent $\iff$ $[\LLL_0,\LLL]=\{0\}$ $\iff$ every homogeneous element of $\LLL$ is $\ad$-locally nilpotent on $\LLL$. In that case, the nilpotency class of $\LLL$ is at most $\max\{2,N_{\Psi}\}$, where $N_{\Psi}$ is the nilpotency class of $\g_{\Psi}$.
\end{enumerate}
\end{theorem}
\begin{proof}
(1)(2) Since $\LLL$ is graded, it admits a triangular decomposition $$\LLL=\LLL^-\oplus\LLL_0\oplus\LLL^+,$$
where $\LLL^{\pm}:=\LLL\cap\nn^{\pm}$ and $\LLL_0:=\LLL\cap\hh$. Moreover, by Lemma~\ref{lemma:real_part_closed}, $$\Psi:=\{\alpha\in\Delta^{re} \ | \ \LLL\cap\g_{\alpha}\neq\{0\}\}$$ is a closed set of real roots. Setting $\LLL^{im\pm}:=\LLL^{\pm}\cap\nn^{im\pm}$, we have
$$\LLL=\LLL_0\oplus\g_{\Psi}\oplus\LLL^{im+}\oplus\LLL^{im-}.$$
Note also that, in view of the $\ad$-local finiteness assumption on the elements of $\LLL$, Theorem~\ref{thm:adx^ny} implies that $\LLL^{im\pm}$ is abelian and that $[\g_{\Psi},\LLL^{im\pm}]=\{0\}$. 

(3) To show that $[\LLL^{im+},\LLL^{im-}]\subseteq \hh\oplus\g_{\Psi}$, assume for a contradiction that there is some $x\in\LLL^{im+}$ and $y\in\LLL^{im-}$ such that $[x,y]$ is a nonzero element of $\nn^{im+}$ (the case where $[x,y]\in\nn^{im-}$ being symmetric, using the action of $\omega$). Let $\alpha:=\deg(x)$ and $\beta:=-\deg(y)$, so that $\alpha,\beta,\alpha-\beta\in\Delta^{im+}$. Since $[x,y]\neq 0$ and $\alpha\neq \beta$, Lemma~\ref{lemma:Kac_free_Lie}(2) implies that $\alpha,\beta$ (and hence also $\alpha$ and $\alpha-\beta$) are not proportional isotropic roots. Moreover, as $\supp(w\alpha-w\beta)\subseteq\supp(w\alpha)$ for all $w\in\WW$, Lemma~\ref{lemma:imaginary_not_connected} implies that $\alpha+(\alpha-\beta)\in\Delta^{im+}$. Hence $(\alpha|\alpha-\beta)<0$ by Lemma~\ref{lemma:basic_propI}(2), so that $[x,[x,y]]\neq 0$ by Theorem~\ref{thm:pre-free}. Therefore, Theorem~\ref{thm:adx^ny} implies that $(\ad x)^n[x,y]\neq 0$ for all $n\in\NN$, a contradiction. 

(4) This readily follows from Proposition~\ref{prop:closed_set_real_roots}.

(5) Note first that $$\LLL^{im}=\LLL^{im+}\oplus\LLL^{im-}\oplus[\LLL^{im+},\LLL^{im-}]$$ by (2) and (3). If $[\LLL^{im},\LLL_0\cap\LLL^{im}]=\{0\}$, then 
$$[\LLL^{im},\LLL^{im}]=[\LLL^{im+},\LLL^{im-}]\subseteq (\LLL_0\cap\LLL^{im})+\g_{\Psi}$$
by (2) and (3), and hence $[\LLL^{im},[\LLL^{im},\LLL^{im}]]=\{0\}$ by (2), that is, $\LLL^{im}$ is nilpotent of degree at most $2$. If $\LLL^{im}$ is nilpotent, then it is solvable. Finally, if $\LLL^{im}$ is solvable, then 
$$[\LLL_0\cap [\LLL^{im},\LLL^{im}],\LLL^{im}]=\{0\}$$ by Lemma~\ref{lemma:solvable_hh}.
Since $\LLL_0\cap\LLL^{im}\subseteq [\LLL^{im+},\LLL^{im-}]\subseteq [\LLL^{im},\LLL^{im}]$, this implies that $[\LLL^{im},\LLL_0\cap\LLL^{im}]=\{0\}$, as desired.

(6) Assume first that $\LLL$ is solvable. Then $\g_{\Psi}$ is a nilpotent subalgebra: otherwise, $\Psi$ contains a pair of opposite roots by (4), and hence $\g_{\Psi}+[\g_{\Psi},\g_{\Psi}]$ contains a copy of $\sll_2(\CC)$, a contradiction. Moreover, as $\LLL^{im}\subseteq\LLL$ is solvable, it is nilpotent by (5). 

Assume next that $\g_{\Psi}$ and $\LLL^{im}$ are nilpotent subalgebras. Then $\LLL^{(1)}=[\LLL,\LLL]\subseteq \g_{\Psi}+\LLL^{im}$ by (1) and (2). Since $[\g_{\Psi},\LLL^{im}]=\{0\}$ by (2), we deduce that $\LLL^{(1)}$ is nilpotent. 
Finally, if $\LLL^{(1)}$ is nilpotent, then $\LLL$ is solvable.

(7) If every homogeneous element of $\LLL$ is $\ad$-locally nilpotent on $\LLL$, then $[\LLL_0,\LLL]=\{0\}$, for if $h\in\LLL_0$ and $x\in\LLL\cap\g_{\alpha}$ ($\alpha\in\Delta$) are such that $[h,x]=\alpha(h)x\neq 0$, then $(\ad h)^nx=\alpha(h)^nx\neq 0$ for all $n\in\NN$. 

Assume next that $[\LLL_0,\LLL]=\{0\}$. In particular, $\Psi$ does not contain any pair of opposite roots (otherwise, if $\pm\alpha$ is such a pair, then $\CC\alpha^{\vee}=\CC [e_{-\alpha},e_{\alpha}]\subseteq\LLL_0$ but $[\alpha^{\vee},e_{\alpha}]=2e_{\alpha}\neq 0$), and hence $\g_{\Psi}$ is a nilpotent subalgebra by (4). Similarly, $\LLL^{im}$ is nilpotent by (5). Since $[\g_{\Psi},\LLL^{im}]=\{0\}$ by (2), it then follows from an easy induction on $n$ that $\LLL^n\subseteq \g_{\Psi}^n+(\LLL^{im})^n$ for all $n\in\NN$, so that $\LLL$ is nilpotent. Moreover, since $(\LLL^{im})^2=\{0\}$ by (5), the nilpotency class of $\LLL$ is at most $\max\{2,N_{\Psi}\}$.

Finally, if $\LLL$ is nilpotent, then of course every homogeneous element of $\LLL$ is $\ad$-locally nilpotent on $\LLL$.
\end{proof}

Note that the spaces $\g_{\Psi}$ for $\Psi$ a closed set of real roots are described in Proposition~\ref{prop:closed_set_real_roots}. The abelian graded subalgebras of $\nn^{im\pm}$, on the other hand, are described in the following proposition.

\begin{prop}\label{prop:abelian_subalgebra_nim+}
Let $\LLL^{im+}$ be a graded subalgebra of $\nn^{im+}$. Then the following assertions are equivalent:
\begin{enumerate}
\item
All elements of $\LLL^{im+}$ are $\ad$-locally finite on $\LLL^{im+}$. 
\item
$\LLL^{im+}$ is abelian. 
\item
There exists some $w^*\in\WW^*$ such that $w^*\LLL^{im+}$ has the form
$$w^*\LLL^{im+}=\CC x_{\beta_1}\oplus\dots\oplus \CC x_{\beta_n}\oplus \LLL_{\delta_1}\oplus\dots\oplus\LLL_{\delta_m},$$
where
\begin{itemize}
\item
$\beta_i\in\Delta^{im+}_{an}\cap K_0$ and $x_{\beta_i}\in\g_{\beta_i}$ for $i=1,\dots,n$;
\item
$\delta_i\in\Delta^{im+}_{is}\cap K_0$ and $\LLL_{\delta_i}$ is a subspace of $\g_{\NN^*\delta_i}$ for $i=1,\dots,m$;
\item
the union of the $m+n$ subdiagrams $$\supp(\beta_1),\dots,\supp(\beta_n),\supp(\delta_1),\dots,\supp(\delta_m)$$ of $\Gamma(A)$ has $m+n$ distinct connected components (namely, the above $m+n$ subdiagrams).
\end{itemize}
\end{enumerate}
\end{prop}
\begin{proof}
We prove that (1)$\implies$(3), the implications (3)$\implies$(2)$\implies$(1) being clear. If (1) holds, then Theorems~\ref{thm:pre-free} and \ref{thm:adx^ny} imply that $(\alpha|\beta)\geq 0$ for all distinct $\alpha,\beta\in\Delta^{im+}$ such that $\LLL^{im+}\cap\g_{\alpha}\neq\{0\}\neq \LLL^{im+}\cap \g_{\beta}$. By Lemma~\ref{lemma:basic_propI}(2), we thus find some $\gamma_1,\dots,\gamma_n\in\Delta^{im+}_{an}$ and some $\gamma_{n+1},\dots,\gamma_{n+m}\in\Delta^{im+}_{is}$ with $\gamma_i+\gamma_j\notin\Delta$ whenever $i\neq j$, such that $$\LLL^{im+}=\CC x_{\gamma_1}\oplus\dots\oplus \CC x_{\gamma_n}\oplus \LLL_{\gamma_1}\oplus\dots\oplus\LLL_{\gamma_m}$$ for some $x_{\gamma_i}\in\g_{\gamma_i}$ and some subspaces $\LLL_{\gamma_i}\subseteq\g_{\NN^*\gamma_i}$. The statement (3) then follows from Lemma~\ref{lemma:imaginary_not_connected}.
\end{proof}

Note also that, since in a nilpotent subalgebra of $\g(A)$, every element is $\ad$-locally nilpotent (hence $\ad$-locally finite), Theorem~\ref{thm:subgA} gives a complete description of nilpotent graded subalgebras of $\g(A)$. In particular, it has the following corollary. Let $N\in\NN$ be as in the statement of Proposition~\ref{prop:closed_set_real_roots}.

\begin{corollary}\label{corollary:nilp}
Let $\LLL$ be a graded subalgebra of $\g(A)$. Then $\LLL$ is nilpotent if and only if every homogeneous $x\in\LLL$ is $\ad$-locally nilpotent on $\LLL$. In that case, the nilpotency class of $\LLL$ is at most $\max\{2,N\}$.
\end{corollary}
\begin{proof}
This readily follows from Proposition~\ref{prop:closed_set_real_roots} and Theorem~\ref{thm:subgA}(7).
\end{proof}

\begin{remark}\label{remark:nilpotency_class}
Let $\LLL$ be a nonabelian nilpotent graded subalgebra of $\g(A)$ (with nilpotency class $N_{\LLL}\geq 2$), and let $\g_{\Psi}$ be as in Theorem~\ref{thm:subgA}. Then Theorem~\ref{thm:subgA}(7) implies that either $\g_{\Psi}$ is abelian and $N_{\LLL}=2$, or else $N_{\LLL}$ coincides with the nilpotency class of $\g_{\Psi}$. In the latter case, the (proof of the) last statement of Proposition~\ref{prop:closed_set_real_roots} then implies that $N_{\LLL}$ is the supremum of the nilpotency classes of the (finite-dimensional) subalgebras $\g_{\Psi'}$ with $\Psi'$ a nilpotent subset of $\Psi$ (in the terminology of \cite{CM18}). A uniform upper bound for the nilpotency class of such $\g_{\Psi'}$ was obtained in \cite{Cap07}.
\end{remark}

As shown by the following example, it is in general \emph{not} true that in a solvable graded subalgebra $\LLL$ of $\g(A)$, every element is $\ad$-locally finite on $\LLL$.
\begin{example}\label{example:solvable_locfin}
Assume that $A=\begin{psmallmatrix}2 & -2 \\ -2 &2\end{psmallmatrix}$, and let $\alpha_1,\alpha_2$ be the simple roots of the corresponding affine Kac--Moody algebra $\g(A)$. Let $\delta:=\alpha_1+\alpha_2\in\Delta^{im+}$. For each $m\in\NN$, set
$$\LLL_{[m]}:=\g_{\delta}\oplus\g_{\Psi_m}, \quad\textrm{where $\Psi_m:=\{n\delta+\alpha_1 \ | \ n\geq m\}\subseteq\Delta^{re+}$}.$$
Then $\g_{\delta}$ and $\g_{\Psi_m}$ are abelian subalgebras of $\g(A)$, and we have $$[\g_{\delta},\g_{n\delta+\alpha_1}]=\g_{(n+1)\delta+\alpha_1}\quad\textrm{for all $n\in\NN$.}$$
In particular, $\LLL:=\LLL_{[0]}$ is a (graded) subalgebra of $\g(A)$. Moreover, $[\LLL,\LLL]=\g_{\Psi_1}$ is abelian, and hence $\LLL$ is solvable (but not nilpotent, as $\LLL^n=\g_{\Psi_n}$ for all $n\in\NN^*$). Note, however, that the nonzero elements of $\g_{\delta}$ are not $\ad$-locally finite on $\LLL$.

Let now $d\in \hh$ be such that $\delta(d)=1$ and $\alpha_1(d)=\alpha_2(d)=0$. Then $\widehat{\LLL}:=\CC d\oplus\LLL$ is also a graded subalgebra of $\g(A)$, such that $\widehat{\LLL}^n=\LLL_{[1]}$ for all $n\in\NN^*$. Hence $\widehat{\LLL}$ is solvable, but $\widehat{\LLL}^1=\LLL_{[1]}$ is not nilpotent. Note, however, that $\widehat{\LLL}^{(2)}=[\LLL_{[1]},\LLL_{[1]}]=\g_{\Psi_2}$ is nilpotent.
\end{example}

Nevertheless, as shown by the following lemma, the (affine) situation described in Example~\ref{example:solvable_locfin} is the only type of obstruction for a solvable graded subalgebra of $\g(A)$ to satisfy the hypothesis of Theorem~\ref{thm:subgA}.

\begin{lemma}\label{lemma:solvable_implies}
Let $\LLL$ be a solvable graded subalgebra of $\g(A)$, and let $x\in\LLL$ be homogeneous, of degree $\alpha\in\Delta^+$. Assume that there exists some homogeneous $y\in\LLL$, of degree $\beta\in\Delta$, such that $(\ad x)^ny\neq 0$ for all $n\in\NN$. Then $(\alpha|\alpha)=0$. Moreover, if $w\in\WW$ is such that $w\alpha\in K_0$, then $\supp(w\beta)\subseteq\supp(w\alpha)$.
\end{lemma}
\begin{proof}
If $\alpha\in\Delta^{re}$, then $x$ is $\ad$-locally nilpotent on $\g(A)$ by Lemma~\ref{lemma:RS}, and hence also on $\LLL$. Thus $\alpha\in\Delta^{im+}$ and hence $(\alpha|\alpha)\leq 0$ by Lemma~\ref{lemma:basic_propI}(1). Up to replacing $y$ with $(\ad x)^ny$ for some large enough $n$, we may moreover assume that $\beta\in\Delta^+$.

If $(\alpha|\alpha)<0$ or if $(\alpha|\beta)<0$, then $(\alpha|\beta+n\alpha)=(\alpha|\beta)+n(\alpha|\alpha)<0$ for some large enough $n\in\NN$. But then $x$ and $(\ad x)^ny\in\LLL$ generate a non-solvable subalgebra of $\LLL$ by Lemma~\ref{lemma:xy_non_solvable}, a contradiction. Thus $(\alpha|\alpha)=0$ and $(\alpha|\beta)\geq 0$.

If $\beta\in\Delta^{im+}$, then Lemma~\ref{lemma:basic_propI}(2) implies that $\alpha,\beta$ are proportional isotropic roots, and hence $[x,y]=0$ by Lemma~\ref{lemma:Kac_free_Lie}(2), a contradiction. Thus $\beta\in\Delta^{re+}$, and the claim follows from Lemma~\ref{lemma:isotropic+real}(1) (note that $\supp(w\alpha)\cup\supp(w\beta)$ is connected, for otherwise $\alpha+\beta\notin\Delta$, contradicting the fact that $[x,y]\neq 0$).
\end{proof}

\begin{corollary}\label{corollary:solvable_non_affine}
Assume that $\Gamma(A)$ does not contain any subdiagram of affine type. Let $\LLL$ be a graded subalgebra of $\g(A)$.
Then the following assertions are equivalent:
\begin{enumerate}
\item
$\LLL$ is solvable.
\item
$[\hh\cap\LLL^1,\LLL]=\{0\}$ and every homogeneous element of $\LLL$ is $\ad$-locally finite on $\LLL$.
\item
$\LLL^1$ is nilpotent.
\end{enumerate}
\end{corollary}
\begin{proof}
The implication (3)$\implies$(1) is clear. If $\LLL$ is solvable, then the homogeneous elements of $\LLL$ are $\ad$-locally finite on $\LLL$ by Lemma~\ref{lemma:solvable_implies} (i.e. by assumption and (\ref{eqn:isotropic}), $\g(A)$ has no isotropic roots), and $[\hh\cap\LLL^1,\LLL]=\{0\}$ by Lemma~\ref{lemma:solvable_hh}, proving (1)$\implies$(2). Finally, assume that (2) holds. Then $\LLL$ has a decomposition $\LLL=\LLL_0\oplus\g_{\Psi}\oplus\LLL^{im+}\oplus\LLL^{im-}$ as in Theorem~\ref{thm:subgA}(1). Moreover, Theorem~\ref{thm:subgA}(4) implies that $\g_{\Psi}$ is a nilpotent subalgebra, because if $\pm\alpha\in\Psi$, then $2e_{\alpha}=[[e_{-\alpha},e_{\alpha}],e_{\alpha}]\in [\hh\cap\LLL^1,\LLL]$, contradicting (2). Similarly, Theorem~\ref{thm:subgA}(5) implies that $\LLL^{im}$ is a nilpotent subalgebra, because $\hh\cap\LLL^{im}\subseteq \hh\cap [\LLL^{im+},\LLL^{im-}]\subseteq \hh\cap\LLL^1$. Therefore, Theorem~\ref{thm:subgA}(6) implies that $\LLL^1$ is nilpotent, proving (2)$\implies$(3).
\end{proof}
As illustrated by Example~\ref{example:solvable_locfin}, the implications (1)$\implies$(2) and (1)$\implies$(3) in Corollary~\ref{corollary:solvable_non_affine} fail as soon as $\Gamma(A)$ contains a subdiagram of affine type. Nevertheless, a weaker form of the equivalence (1)$\iff$(3) can still be proved in general, as shown by the following theorem.

\begin{theorem}\label{thm:solvable_general}
Let $\LLL$ be a graded subalgebra of $\g(A)$. Then $\LLL$ is solvable if and only if $\LLL^{(2)}$ is nilpotent.
\end{theorem}
\begin{proof}
If $\LLL^{(2)}$ is nilpotent, then it is solvable and hence $\LLL$ is solvable. Assume now that $\LLL$ is solvable. To prove that $\LLL^{(2)}$ is nilpotent, it is sufficient to prove by Corollary~\ref{corollary:nilp} that every homogeneous element of $\LLL^{(2)}$ is $\ad$-locally nilpotent on $\LLL^{(2)}$. Assume for a contradiction that there exist some $x,y\in \LLL^{(2)}$ such that $(\ad x)^ny\neq 0$ for all $n\in\NN$. 

Set $\alpha:=\deg(x)$ and $\beta:=\deg(y)$. Note that $\alpha,\beta\in\Delta$ by Lemma~\ref{lemma:solvable_hh}, and hence $\alpha\in\Delta^{im}$ by Lemma~\ref{lemma:RS}. Up to using the action of $\omega$, we may assume that $\alpha\in\Delta^{im+}$. Lemma~\ref{lemma:solvable_implies} then implies that $\alpha\in\Delta_{is}^{im+}$. Up to using the $\WW^*$-action, we may then assume that $\alpha\in K_0$, and hence that $\supp(\alpha)$ is of affine type by (\ref{eqn:isotropic}).

Since $x\in \LLL^{(2)}$, there exist some homogeneous $x_1,x_2\in\LLL^{1}$ with $[x_1,x_2]\neq 0$ such that $\gamma_1+\gamma_2=\alpha$, where $\gamma_1:=\deg(x_1)$ and $\gamma_2:=\deg(x_2)$. Note that $\gamma_1,\gamma_2\in\Delta$ by Lemma~\ref{lemma:solvable_hh}. Moreover, Lemma~\ref{lemma:xy_non_solvable} yields
\begin{equation}\label{eqn:gageqzero}
(\gamma_1|\alpha)\geq 0\quad\textrm{and}\quad(\gamma_2|\alpha)\geq 0.
\end{equation}

If $\gamma_i\in\Delta^{re}$ for some $i\in\{1,2\}$, then (\ref{eqn:gageqzero}) and Lemma~\ref{lemma:isotropic+real}(1) imply that $\supp(\gamma_i)\subseteq\supp(\alpha)$ (note that $\supp(\gamma_i)\cup\supp(\alpha)$ is connected as $\gamma_1+\gamma_2=\alpha$). Since $\gamma_1+\gamma_2=\alpha$, we then have $\supp(\gamma_1),\supp(\gamma_2)\subseteq\supp(\alpha)$, and hence $\gamma_1,\gamma_2\in\Delta^{re}$ by (\ref{eqn:isotropic_multiple}). Moreover, $(\gamma_1|\alpha)=0=(\gamma_2|\alpha)$ by (\ref{eqn:isotropic_zero}), contradicting Lemma~\ref{lemma:solvable_affine}. Therefore, $\gamma_1,\gamma_2\in\Delta^{im}$.

Up to permuting $\gamma_1$ and $\gamma_2$, we may assume that $\gamma_1\in\Delta^{im+}$. If $\gamma_2\in\Delta^{im+}$, then $\supp(w\gamma_1)\subseteq\supp(w\alpha)$ for all $w\in\WW$ (because $w\alpha\in\Delta^{im+}$ is the sum of $w\gamma_1\in\Delta^{im+}$ and $w\gamma_2\in\Delta^{im+}$), and if $\gamma_2\in\Delta^{im-}$, then $\supp(w\alpha)\subseteq\supp(w\gamma_1)$ for all $w\in\WW$ (because $w\gamma_1\in\Delta^{im+}$ is the sum of $w\alpha\in\Delta^{im+}$ and $-w\gamma_2\in\Delta^{im+}$). In both cases, Lemma~\ref{lemma:imaginary_not_connected} implies that $\alpha+\gamma_1\in\Delta^{im+}$. We then deduce from (\ref{eqn:gageqzero}) and Lemma~\ref{lemma:basic_propI}(2) that $\alpha$ and $\gamma_1$ are proportional isotropic roots. But then $\gamma_1,\gamma_2$ are also proportional isotropic roots, so that $[x_1,x_2]=0$ by Lemma~\ref{lemma:Kac_free_Lie}(2), a contradiction.
\end{proof}

We conclude this section by illustrating Theorem~\ref{thm:subgA} and Proposition~\ref{prop:abelian_subalgebra_nim+} with some examples of nilpotent graded subalgebras of $\g(A)$.

\begin{example}
Let $\Psi\subseteq\Delta^{re}$ be a closed set of real roots such that $\Psi\cap-\Psi=\varnothing$. Then $\LLL:=\g_{\Psi}$ is a nilpotent graded subalgebra of $\g(A)$ by Proposition~\ref{prop:closed_set_real_roots}.
\end{example}

\begin{example}
Let $\delta\in\Delta_{is}^{im+}$, and let $\LLL_{\delta}$ be a graded subalgebra of the (abelian) Lie algebra $\g_{\NN^*\delta}$. Let also $\Psi\subseteq\Delta^{re+}$ be a closed set of roots such that $[\g_{\Psi},\LLL_{\delta}]=\{0\}$. Then $\LLL:=\g_{\Psi}\oplus\LLL_{\delta}$ is a nilpotent graded subalgebra of $\g(A)$.

For instance, if $A$ is of type $A_{n-1}^{(1)}$ for some $n\geq 4$, then $\g(A)$ is a double $1$-dimensional extension of $\sll_n(\CC[t,t\inv])$ (see \cite[Theorem~7.4]{Kac}). Writing $E_{ij}$ for the matrix of $\sll_n(\CC[t,t\inv])$ with a ``$1$'' in position $(i,j)$ and ``$0$'' elsewhere, one could then take $\g_{\Psi}=\CC E_{12}$ and $\LLL_{\delta}=\CC (tE_{33}-tE_{44})$.
\end{example}

\begin{example}
Let $\beta\in\Delta^{im+}_{an}$ and $i\in I$ be such that $\dim\g_{\beta+\alpha_i}<\dim\g_{\beta}$. Then there exists some nonzero $x\in\g_{\beta}$ such that $[e_i,x]=0$. In particular, $\LLL:=\g_{\Psi}\oplus\CC x$ is a nilpotent graded subalgebra of $\g(A)$, where $\Psi:=\{\alpha_i\}$.
\end{example}

\begin{example}
Assume that $A=\begin{psmallmatrix}2&-2&-1\\ -2 &2& -1\\ -1&-1&2\end{psmallmatrix}$. Consider the nonzero elements 
$$y:=[e_3,[e_2,e_1]]+2[e_2,[e_3,e_1]]\quad\textrm{and}\quad x:=s_1^*y=[[e_1,e_3],[e_1,e_2]]+[e_3,[e_1,[e_2,e_1]]]$$
of $\g(A)$ (see \cite[Equation (4.7) on page 69]{KMGbook}). Then $[f_1,y]=0$, and hence also $[e_1,x]=s_1^*[f_1,y]=0$. Note, moreover, that $$y^*:=\omega(y)=-[f_3,[f_2,f_1]]-2[f_2,[f_3,f_1]]\in\nn^{im-}\quad\textrm{and}\quad x\in\nn^{im+}.$$
Finally, a straightforward computation yields that
$$[y^*,x]=-24 e_1.$$
Consider the closed set $\Psi:=\{\alpha_1\}\subseteq\Delta^{re}$ and set $\LLL^{im+}:=\CC x$ and $\LLL^{im-}:=\CC y^*$. Then 
$$\LLL:=\g_{\Psi}\oplus\LLL^{im+}\oplus\LLL^{im-}$$
is a graded subalgebra of $\g(A)$ isomorphic to the three-dimensional Heisenberg algebra. In particular, $\LLL$ is nilpotent of degree $2$, whereas $\g_{\Psi}$ is abelian: this shows that the upper bound on the nilpotency class provided in Theorem~\ref{thm:subgA}(7) cannot be improved.
\end{example}

\bibliographystyle{amsalpha} 
\bibliography{these}

\end{document}